\documentclass[12pt]{amsart}
\usepackage[doublespacing]{setspace}
\usepackage[utf8]{inputenc}
\usepackage{graphicx}
\usepackage{amscd}
\usepackage{amsmath}
\usepackage{enumerate}
\usepackage{amssymb}
\usepackage{amsthm}
\usepackage{amsfonts}
\usepackage{fancyhdr}
\usepackage{mathrsfs}
\usepackage{amsxtra}

\usepackage{setspace}
\numberwithin{equation}{section}
\newtheorem{theorem}[equation]{Theorem}
\newtheorem{lemma}[equation]{Lemma}
\newtheorem{corollary}[equation]{Corollary}
\newtheorem{proposition}[equation]{Proposition}

\newtheorem{example}[equation]{Example}
\newtheorem{remark}[equation]{Remark}

\def\beq{\begin{eqnarray*}}
\def\eeq{\end{eqnarray*}}
\def\bt{\begin{theorem}}
\def\et{\end{theorem}}
\def\bp{\begin{proposition}}
\def\ep{\end{proposition}}
\def\bl{\begin{lemma}}
\def\el{\end{lemma}}

\def\bp{{\bar\partial}}

\def\HH{{\mathcal H}}

\def\RR{\mathcal{R}}

\def\XX{{\mathcal{X}}}

\def\span{\text{\rm span}}

\def\dist{\text{\rm {dist}}}
\def\supp{\text{\rm supp}\,}

\def\dist{\rm dist}

\usepackage{color}

\def\UK{{\mathcal U\mathcal K}}

\def\Z{{\mathbf{Z}}}

\def\R{{\mathbb R}}

\title{{On Asymptotically Orthonormal Sequences }}

\author[Fricain]{Emmanuel Fricain}
\address{Laboratoire Paul Painlev\'e, Universit\'e Lille 1, 59 655 Villeneuve d'Ascq C\'edex }
\email{emmanuel.fricain@math.univ-lille1.fr}

\author[Rupam]{Rishika Rupam}
\address{Laboratoire Paul Painlev\'e, Universit\'e Lille 1, 59 655 Villeneuve d'Ascq C\'edex}
\email{rishika.rupam@math.univ-lille1.fr}

\thanks{The authors were supported by Labex CEMPI (ANR-11-LABX-0007-01)}

\keywords{Asymptotically orthormal sequences, de Branges--Rovnyak and model spaces, reproducing kernels.}

\subjclass[2010]{30J05, 30H10, 46E22}

\begin{document}

\begin{abstract}
An asymptotically orthonormal sequence is a sequence which is "nearly" orthonormal in the sense that it satisfies the Parseval equality up to two constants close to one. In this paper, we explore such sequences formed by normalized reproducing kernels of model spaces and de Branges--Rovnyak spaces.
\end{abstract}

\maketitle


\section{Introduction}
When working in Hilbert spaces, it is very natural and useful to deal with orthonormal basis. However, in many situations, the system we are interested in does not form an orthonormal basis but is close to one. The investigation of such bases has a long history.  It started with the works of Paley--Wiener \cite{paley1934fourier} and Levinson \cite{levinson1940gap} mainly for exponential systems. In this context, functional models have been used in \cite{hruvsvcev1981unconditional} together with some other tools from operator theory. The model spaces $K_\Theta$ of the unit disc are subspaces of the Hardy space $H^2(\mathbb{D})$ invariant under the adjoints of multiplications. Their theory is connected to dilation theory for contractions on Hilbert spaces. The paper \cite{hruvsvcev1981unconditional} has inspired a fruitful line of research on geometric properties of systems formed by reproducing kernels of $K_\Theta$. Not only it enabled to recapture all the classical results on exponential systems but also it gave many new results in a more general context. 

In \cite{chalendar2003functional}, along the line of research of \cite{hruvsvcev1981unconditional}, the authors studied the case when the system of normalized reproducing kernels 
$(\kappa_{\lambda_n}^\Theta)_n$ of $K_\Theta$ is asymptotically close to an orthonormal basis (see definition below). This is a particular case of unconditional basis where more rigidity is required. It should be noted that in \cite{hruvsvcev1981unconditional} and \cite{chalendar2003functional}, the additional assumption 
\begin{equation}\label{eq:additional-intro-hruvsvcev1981unconditional}
\sup_{n\geq 1}|\Theta(\lambda_n)|<1
\end{equation}
is required. Under that assumption, the projection method developed in \cite{hruvsvcev1981unconditional} and used in \cite{chalendar2003functional} was to link the properties of $(\kappa_{\lambda_n}^\Theta)_n$ to those of normalized reproducing kernels $(\kappa_{\lambda_n})_n$ of $H^2(\mathbb D)$. Volberg proved in \cite{volberg1982two} that $(\kappa_{\lambda_n})_n$ is an asymptotically orthonormal basis of its closed span if and only if $(\lambda_n)_n$ is a thin sequence (a stronger condition than Carleson's condition). This beautiful result was recently reproved by Gorkin--McCarthy--Pott--Wick \cite{gorkin2014thin} by a direct and easier method using ideas from interpolation theory.

Following the work of Baranov \cite{baranov2005stability} for Riesz basis, we are interested in this paper to investigate asymptotically orthonormal basis of reproducing kernels of $K_\Theta$ without requiring assumption \eqref{eq:additional-intro-hruvsvcev1981unconditional}. In that situation, the projection method does no longer applies and the basic tool here will be Bernstein's type inequalities. We also work in the more general context where model spaces $K_\Theta$ are replaced by de Branges--Rovnyak spaces $\HH(b)$. We should mention that we work in this paper in the upper--half plane but most results transfer easily in the unit disc. 

The plan of the paper is the following. The next section contains preliminary material ; in particular, an analogue of Bari's theorem is given, which completes a result given in \cite{chalendar2003functional}. In Section 3, we study the stability of asymptotically orthonormal sequences with respect to perturbation of  frequencies. The main results of the paper are Theorem~\ref{prop1}, Corollary~\ref{cor1-perturbation}, Theorem~\ref{Thm:stability-real-frequencies} and Corollary~\ref{Cor:MIF-CLS-Stability-result}. In Section 4, we study the case of exponential systems. Finally, in the last section, we examine what happens when one projects an AOB 
$(\kappa_{\lambda_n}^{b_1})_{n\geq 1}$ on a subspace $\HH(b_2)$ of $\HH(b_1)$.
\\ \\
\noindent \textbf{Acknowledgement} The authors would like to thank Pascal Thomas for useful discussions 
and suggesting the problem discussed in Section 5.

\section{Preliminaries}
\subsection{Asymptotically orthonormal sequences}
Let $\HH$ be a Hilbert space, $\XX=(x_n)_{n\geq 1}$ be a sequence of vectors in $\HH$. We recall that $\XX$ is said to be:
\begin{enumerate}
\item[(a)] \emph{minimal} if for every $n\geq 1$, 
$$
x_n\not\in\span(x_\ell:\ell\not=n),
$$
where $\span(\dots)$ denotes the closure of the finite linear combination of $(\dots)$. 
\item[(b)] a \emph{Riesz sequence} (abbreviated RS) if there exists two constants $c,C>0$ such that 
$$ 
c\sum_{n\geq 1}|a_n|^2\leq \|\sum_{n\geq 1}a_n x_n\|^2_{\HH}\leq C\sum_{n\geq 1}|a_n|^2,
$$
for every finitely supported complex sequence $(a_n)_n$;
\item[(c)] an \emph{asymptotically orthonormal sequence} (abbreviated AOS) if there exists $N_0\in\mathbb{N}$ such that for all $N\geq N_0$ there are constants $c_N,C_N>0$ verifying
\begin{equation}\label{eq:AOB}
c_N\sum_{n\geq N}|a_n|^2\leq \|\sum_{n\geq N}a_n x_n\|^2_{\HH}\leq C_N\sum_{n\geq N}|a_n|^2,
\end{equation}
for every finitely supported complex sequence $(a_n)_n$ and $\lim_{N\to \infty}c_N=1=\lim_{N\to \infty}C_N$;
\item[(d)] an \emph{asymptotically orthonormal basic sequence} (abbreviated AOB) if it is an AOS with $N_0=1$; 
\item[(e)] a \emph{Riesz basis} (abbreviated RB) if it is a complete Riesz sequence, that is a Riesz sequence satisfying 
$$
\span(x_n:n\geq 1)=\HH.
$$
 \end{enumerate}
It is easy to see that $(x_n)_{n\geq 1}$ is an AOB if and only if it is an AOS as well as a RS. Also, $(x_n)_{n\geq 1}$ is an AOB if and only if it is minimal and an AOS. The well-known result of K\"othe--Toeplitz (\cite[page 136]{nikolski1986treatise}) says that if $\XX=(x_n)_{n\geq 1}$ is a complete and minimal sequence of normalized  vectors in $\HH$, then $\XX$ is a Riesz basis if and only if $\XX$ is an unconditional basis. The reader should pay attention to the fact that AOB does not imply completeness; an AOB is a basis in its span but not necessarily in the whole space.

We recall also that for a sequence $\XX=(x_n)_{n\geq 1}$, the \emph{Gram matrix} $\Gamma_{\XX}=(\Gamma_{n,p})_{n,p\geq 1}$ is defined by 
$$
\Gamma_{n,p}=\langle x_n,x_p \rangle_{\HH},\qquad (n,p\geq 1).
$$
If $\XX=(x_n)_{n\geq 1}$ is a complete and minimal sequence and $\XX^*=(x_n^*)_{n\geq 1}$ is its biorthogonal, that is the unique sequence $(x_n^*)_{n\geq 1}$ in $\HH$ satisfying 
$$
\langle x_\ell,x_n^* \rangle_\HH=\delta_{n,\ell},
$$
the \emph{interpolation operator} $J_{\mathcal X}$ is defined as 
$$
J_{\mathcal X}x=(\langle f,x_n^*\rangle_{\HH})_{n\geq 1},\qquad (x\in\HH).
$$
We refer the reader to \cite{nikolski1986treatise}, \cite{fricain2014theory} or \cite{young2001introduction} for details on general geometric properties of sequences ib an Hilbert space.  

Bari's theorem (see \cite[page 132]{nikolski1986treatise}) gives several caracterizations for a sequence to be a RB in terms of its Gram matrix and interpolation operator. An analogue of Bari's result for AOB is also available. A part of this can be found in \cite{chalendar2003functional}. To complete the picture, we need two preliminaries results. First we introduce a notation. Let $T\in\mathcal L(\HH_1,\HH_2)$. We say that $T\in \UK(\HH_1,\HH_2)$ if $T$ is invertible from $\HH_1$ onto $\HH_2$, $T=U+K$ where $U,K\in \mathcal L(\HH_1,\HH_2)$ and $U$ is unitary and $K$ is compact.

\bl\label{Lem:unitary-compact}
Let $\HH_1,\HH_2,\HH_3$ be Hilbert spaces. The following hold:
\begin{enumerate}
\item[\rm{(a)}] if $T_1\in\UK(\HH_1,\HH_2)$ and $T_2\in\UK(\HH_2,\HH_3)$, then $T_2T_1\in\UK(\HH_1,\HH_3)$;
\item[\rm{(b)}] if $T\in\UK(\HH_1,\HH_2)$, then $T^{-1}\in\UK(\HH_2,\HH_1)$; 
\item[\rm{(c)}] if $T\in\UK(\HH_1,\HH_2)$, then $T^*\in\UK(\HH_2,\HH_1)$.
\end{enumerate}
\el
\begin{proof}
The proofs of (a) and (c) are straightforward and are left to the reader. Let us prove (b). Assume that $T=U+K$ is invertible with $U$ unitary and $K$ compact. Then, write $T=U(I+U^*K)=UV$ with $V=I+U^*K$. It is clear that $V$ is invertible and $I=V^{-1}+V^{-1}U^*K$. Hence $V^{-1}=I-V^{-1}U^*K$. We then get
$$
T^{-1}=V^{-1}U^*=U^*-V^{-1}U^*K U^*,
$$
which means that $T^{-1}\in\UK(\HH_2,\HH_1)$.
\end{proof}

\bl\label{Lem:NCBAO-Interp}
Let $\XX=(x_n)_{n\geq 1}$ be a complete AOB in $\HH$ and let $C_N$ be the constant appearing in the right inequality of \eqref{eq:AOB}. Then for every $N\geq 1$ and $f\in\HH$, we have
$$
\sum_{n\geq N}|\langle f,x_n\rangle_\HH|^2\leq C_N \|f\|_\HH^2.
$$
\el 

\begin{proof}
Let us denote by $P_N:\ell^2\rightarrow \ell^2$ the orthogonal projection defined by
$$
P_N\left(\sum_{n\geq 1}a_n e_n\right)=\sum_{n\geq N}a_n e_n,
$$
where $(e_n)_{n\geq 1}$ is the canonical orthonormal basis of $\ell^2$. For any $a=(a_n)_{n\geq 1}\in\ell^2$, define 
$$
V_{\mathcal X}a=\sum_{n\geq 1}a_n x_n.
$$
Since $\XX$ is in particular a Riesz basis, this map $V_{\XX}$ defines a continuous invertible operator from $\ell^2$ onto $\HH$. Moreover, for $a\in\ell^2$, we have
$$
\|V_\XX P_N a\|_\HH^2=\|\sum_{n\geq N}a_n x_n\|_\HH^2\leq C_N \sum_{n\geq N}|a_n|^2\leq C_N \|a\|_{\ell^2}^2,
$$
which gives that $\|P_N V_\XX^*\|=\|V_\XX P_N\|\leq \sqrt{C_N}$. But is is easy to see that $V_\XX^*=J_{\XX^*}$, whence $\|P_N J_{\XX^*}\|\leq \sqrt{C_N}$, which gives the desired inequality. 

\end{proof}

\begin{proposition}\label{Bari-AOB}
Let $\XX=(x_n)_{n\geq 1}$ be a complete and minimal sequence of vectors in $\HH$, $\XX^*=(x_n^*)_{n\geq 1}$ its biorthogonal. The following assertions are equivalent:
\begin{enumerate}[(i)]
\item the sequence $\XX$ is an AOB in $\HH$;
\item there exists an operator $U_{\XX}\in\UK(\HH,\ell^2)$ such that $U_{\XX}(x_n)=e_n$, $n\geq 1$; 
\item the Gram matrix defines a bounded invertible operator on $\ell^2$ of the form $I+K$ with $K$ compact;
\item $J_{\XX^*}\in\UK(\HH,\ell^2)$;
\item the sequence $\XX^*$ is a complete AOB in $\HH$;
\item there exists an invertible operator 
$U_{\XX}:\HH\longrightarrow \ell^2$ such that 
$U_{\XX}(x_n)=e_n$, $n\geq 1$, and if 
$U_{\XX,N}:\span(x_n:n\geq N)\longrightarrow \span(e_n:n\geq N)$ is the restriction of $U_{\XX}$ to $\span(x_n:n\geq N)$, then 
$$
\lim_{N\to \infty}\|U_{\XX,N}\|=1=\lim_{N\to \infty}\|U_{\XX,N}^{-1}\|;
$$
\item for every $N\geq 1$, there are two constants $C_N,C_N^*>0$ such that 
\begin{equation}\label{eq:interpola-operator-1}
{C_N^*}^{-1} \|f\|_\HH^2\leq \sum_{n\geq N}|\langle f,x_n\rangle_\HH|^2\leq C_N \|f\|_\HH^2
\end{equation}
for every $f\in\HH\ominus\span(x_1,x_2,\dots,x_{N-1})$ and $\lim_{N\to\infty}C_N=1=\lim_{N\to\infty}C_N^*$;
\item the sequence $\XX^*$ is complete in $\HH$ and for every $N\geq 1$, there are two constants $C_N,C^*_N>0$ such that 
\begin{equation}\label{eq:interpola-operator-2}
\sum_{n\geq N}|\langle f,x_n\rangle_\HH|^2\leq C_N \|f\|_\HH^2 \quad \mbox{and}\quad \sum_{n\geq N}|\langle f,x^*_n\rangle_\HH|^2\leq C^*_N \|f\|_\HH^2,
\end{equation}
for every $f\in\HH$ and $\lim_{N\to\infty}C_N=1=\lim_{N\to\infty}C^*_N$
\end{enumerate}
\end{proposition}

\begin{proof}
The equivalences between $(i)$, $(ii)$ and $(iii)$ are contained in \cite[Proposition 3.2]{chalendar2003functional}. The equivalence with $(iv)$ follows from Bari's theorem, the fact that $J_{\XX^*}=V_\XX^*=(U_\XX^{-1})^*$, and then, by Lemma~\ref{Lem:unitary-compact}, $U_\XX\in \UK(\HH,\ell^2)$ if and only if $J_{\XX^*}\in\UK(\HH,\ell^2)$. 
Let us now prove the others implications.\

$(ii)\Longrightarrow (v)$: since 
$$
\delta_{n,\ell}=\langle U_\XX x_n,U_\XX x_\ell\rangle_{\ell^2}=\langle x_n,U_{\XX}^*U_\XX x_\ell\rangle_\HH,
$$
we get that $x_\ell^*=U_{\XX}^*U_\XX x_\ell=U_\XX^* e_\ell$. In other words, $U_{\XX^*}=(U_\XX^*)^{-1}$ and $\XX^*$ is a complete and minimal sequence. Now $(v)$ follows from Lemma~\ref{Lem:unitary-compact} and the implication $(ii)\Longrightarrow (i)$ applied to $\XX^*$.

$(v)\Longrightarrow (i)$: use the implication $(i)\Longrightarrow (v)$ applied to $\XX^*$.

$(i)\Longrightarrow (vi)$:  by Bari's theorem, we know that $U_\XX$ is a bounded invertible operator from $\HH$ onto $\ell^2$. Moreover, for any $x=\sum_{n\geq N}a_n x_n$, we have 
$$
\|U_{\XX,N}x\|_{\ell^2}^2=\left\|\sum_{n\geq N}a_n e_n\right\|_{\ell^2}^2=\sum_{n\geq N}|a_n|^2,
$$
and using \eqref{eq:AOB}, we have 
$$
c_N \|U_{\XX,N}x\|_{\ell^2}^2\leq \|x\|_\HH^2\leq C_N \|U_{\XX,N}x\|_{\ell^2}^2.
$$
Hence $C_N^{-1/2}\leq \|U_{\XX,N}\|\leq c_N^{-1/2}$. In particular, $\|U_{\XX,N}\|\to 1$ as $N$ goes to $\infty$. Similarly we prove that $\|U_{\XX,N}^{-1}\|\to 1$ as $N$ goes to $\infty$. 

$(vi)\Longrightarrow (i)$: by Bari's theorem, we know that $\XX$ is a Riesz basis. Moreover, we have 
$$
\left\|\sum_{n\geq N}a_n x_n\right\|_\HH^2=\left\|U_{\XX,N}^{-1}\left(\sum_{n\geq N}a_n e_n\right)\right\|_\HH^2\leq \|U_{\XX,N}^{-1}\|^2 \sum_{n\geq N}|a_n|^2,
$$
and 
$$
\sum_{n\geq N}|a_n|^2=\left\|U_{\XX,N}\left(\sum_{n\geq N}a_n x_n\right) \right\|_{\ell^2}^2\leq \|U_{\XX,N}\|^2 \left\|\sum_{n\geq N}a_nx_n\right\|_\HH^2.
$$
Hence 
$$
\|U_{\XX,N}\|^{-2}\sum_{n\geq N}|a_n|^2\leq \left\|\sum_{n\geq N}a_n x_n\right\|_\HH^2\leq \|U_{\XX,N}^{-1}\|^2 \sum_{n\geq N}|a_n|^2.
$$
Since $\|U_{\XX,N}\|$ and $\|U_{\XX,N}^{-1}\|$ go to $1$ as $N$ goes to $\infty$, we get that $(x_n)_{n\geq 1}$ is an AOB in $\HH$.

$(i)\Longrightarrow (vii)$: the right inequality in \eqref{eq:interpola-operator-1} follows from Lemma~\ref{Lem:NCBAO-Interp}. Since $(x_n^*)_{n\geq 1}$ is also a complete AOB in $\HH$, for any $N\geq 1$, there are two positive constants $c_N^*,C_N^*$ satisfying 
$$
c_N^*\sum_{n\geq N}|a_n|^2\leq \left\|\sum_{n\geq N}a_n x_n^*\right\|_{\HH}^2\leq C_N^*\sum_{n\geq N}|a_n|^2,
$$
and $c_N^*,C_N^*$ go to $1$ as $N$ goes to $\infty$. Moreover, we can decompose any $f\in\HH\ominus\span(x_1,x_2,\dots,x_{N-1})$ as 
$$
f=\sum_{n\geq N}\langle f,x_n\rangle_{\HH}x_n^*, 
$$
which gives that 
$$
\|f\|_\HH^2\leq C_N^* \sum_{n\geq N}|\langle f,x_n\rangle_\HH|^2,
$$
and then the second inequality of \eqref{eq:interpola-operator-1}.

$(vii)\Longrightarrow (v)$: since 
$$
{c_1^*}^{-1}\|f\|_\HH^2\leq \sum_{n\geq 1}|\langle f,x_n\rangle_\HH|^2\leq C_1 \|f\|_\HH^2\,
$$
for every $f\in\HH$, the operator $J_{\XX^*}$ is invertible from $\HH$ onto $\ell^2$. Hence, according to Bari's theorem, we know that $\XX$ and $\XX^*$ are Riesz basis for $\HH$. Moreover, every $f=\sum_{n\geq N}a_n x_n^*$ with $(a_n)_{n\geq N}\in\ell^2$ satisfies $f\in\HH\ominus\span(x_1,\dots,x_{N-1})$ and $\langle f,x_k \rangle_\HH=a_k$, $k\geq N$. Hence by \eqref{eq:interpola-operator-1}, we have
$$
{C_N^*}^{-1} \left\|\sum_{n\geq N}a_n x_n^*\right\|_\HH^2\leq \sum_{n\geq N}|a_n|^2\leq C_N \left\|\sum_{n\geq N}a_n x_n^*\right\|_\HH^2.
$$
That proves that $(x_n^*)_{n\geq 1}$ is an AOB. 

$(i)\Longrightarrow (viii)$: follows immediately from Lemma~\ref{Lem:NCBAO-Interp} and the fact that $(i)\Longrightarrow (v)$. 

$(viii)\Longrightarrow (i)$: let $f=\sum_{n\geq N}a_n x_n$, where $(a_n)_{n\geq N}$ is any finite sequence of complex numbers. Then, applying the second inequality of  \eqref{eq:interpola-operator-2} gives that 
$$
\sum_{n\geq N}|a_n|^2\leq C_N^* \left\|\sum_{n\geq N}a_n x_n\right\|_\HH^2.
$$
On the other hand, by Cauchy-Schwarz inequality 
\begin{eqnarray*} 
\left\|\sum_{n\geq N}a_n x_n \right\|^2_\HH & =&\sup_{\substack{g\in\HH\\ \|g\|_\HH\leq 1}} \left|\langle \sum_{n\geq N}a_n x_n,g \rangle_\HH \right|^2
= \sup_{\substack{g\in\HH\\ \|g\|_\HH\leq 1}} \left|\sum_{n\geq N}a_n \langle  x_n,g \rangle_\HH \right|^2\\
&\leq & \sum_{n\geq N}|a_n|^2 \sup_{\substack{g\in\HH\\ \|g\|_\HH\leq 1}} \sum_{n\geq N}|\langle x_n,g\rangle_\HH|^2 \leq C_N \sum_{n\geq N}|a_n|^2.
\end{eqnarray*}

\end{proof}

We now give two simple conditions, one necessary the other one sufficient to be an AOB. 
\begin{proposition}\label{cns-general-AOB}
Let $\XX=(x_n)_{n\geq 1}$ be a sequence of normalized vectors in $\HH$ and let $\Gamma_\XX=(\Gamma_{n,p})_{n,p\geq 1}$ be its Gram matrix. The followings hold:
\begin{enumerate}[(a)]
\item If 
$$
\lim_{N\to \infty}\left(\sup_{n\geq N}\sum_{\substack{p\geq N\\p\neq n}}|\Gamma_{n,p}|\right)=0,
$$
then $(x_n)_{n\geq 1}$ is an AOS.
\item If $(x_n)_{n\geq 1}$ is an AOB then 
$$
\lim_{n\to \infty}\left(\sum_{\substack{p\geq 1\\p\neq n}} |\Gamma_{n,p}|^2\right)=0.
$$
\end{enumerate}
\end{proposition}

\begin{proof} (a) Let $(a_n)_{n\geq 1}$ be a finite sequence of complex numbers and denote by 
$$
\epsilon_N=\sup_{n\geq N}\sum_{\substack{p\geq N\\p\neq n}}|\Gamma_{n,p}|.
$$
Write 
\begin{eqnarray*}
\left\|\sum_{n\geq N}a_n x_n\right\|^2_{\HH} & =&\sum_{n,p\geq N}a_n\overline{a_p} \langle x_n,x_p \rangle_{\HH}\\
&= &\sum_{n,p\geq N}a_n\overline{a_p}\,\Gamma_{n,p}\\
&=&\sum_{n\geq N}|a_n|^2+\sum_{\substack{n,p\geq N\\n\neq p}}a_n\overline{a_p}\,\Gamma_{n,p}.
\end{eqnarray*}
We will prove that 
\begin{equation}\label{eq:lemma-cns-general-AOB}
\left|\sum_{\substack{n,p\geq N\\n\neq p}}a_n\overline{a_p}\,\Gamma_{n,p}\right|\leq \varepsilon_N \sum_{n\geq N}|a_n|^2.
\end{equation}
Using that $ab\leq (a^2+b^2)/2$ and $|\Gamma_{n,p}|=|\Gamma_{p,n}|$, we have 
\begin{eqnarray*}
\left|\sum_{\substack{n,p\geq N\\n\neq p}}a_n\overline{a_p}\,\Gamma_{n,p}\right|&\leq& \frac{1}{2}\sum_{\substack{n,p\geq N\\n\neq p}}(|a_n|^2+|a_p|^2)|\Gamma_{n,p}| \\
&=&\sum_{\substack{n,p\geq N\\n\neq p}} |a_n|^2|\Gamma_{n,p}|\\
&=&\sum_{n\geq N}|a_n|^2 \sum_{\substack{p\geq N\\p\neq n}}|\Gamma_{n,p}|,
\end{eqnarray*}
which gives immediately \eqref{eq:lemma-cns-general-AOB}. Therefore, 
$$
(1-\varepsilon_N)\sum_{n\geq N}|a_n|^2\leq \left\|\sum_{n\geq N}a_n x_n\right\|^2_{\HH} \leq (1+\varepsilon_N)\sum_{n\geq N}|a_n|^2.
$$
Since $\varepsilon_N\to 0$ as $N\to \infty$, the last inequalities give that $(x_n)_{n\geq 1}$ is an AOS. 

(b) Since $\XX=(x_n)_{n\geq 1}$ is an AOB, we know from Proposition~\ref{Bari-AOB} that $\Gamma_\XX=I+K$, with $K$ compact. In particular, we have
$$
\|(\Gamma_\XX-I)e_n\|_{\ell^2}^2=\|K e_n\|_{\ell^2}^2\to 0,\mbox {as }n\to \infty.
$$
It remains to notice that 
$$
\|(\Gamma_\XX-I)e_n\|_{\ell^2}^2=\sum_{\substack{p\geq 1\\p\neq n}} |\Gamma_{n,p}|^2.
$$
\end{proof}

We end this subsection by two stability results. The first one is inspired by an analogue result of Baranov for Riesz basis property \cite{baranov2005stability}. The second one is a generalization of a result appearing in \cite[Proposition 3.3]{chalendar2003functional}.

\begin{proposition}\label{Prop:stability23}
Let $(x_n)_{n\geq 1}$ be an AOS in $\HH$ and let $(x'_n)_{n\geq 1}$ be a sequence of vectors in $\HH$. Suppose there exists $N_0\in\mathbb N$ such that for all $N\geq N_0$ there is $\varepsilon_N>0$ verifying 
\begin{equation}\label{frame}
\displaystyle\sum_{n \geq N} |\langle x, x_n-x'_n\rangle|^2 \leq \varepsilon_N \|x\|_\HH^2,
\end{equation}
for every $x\in\HH$ and $\lim_{N\to \infty}\varepsilon_N=0$. Then $(x'_n)_{n\geq 1}$ is an AOS. Furthermore, if $(x_n)_{n\geq 1}$ is a complete AOB, $N_0=1$ and $\varepsilon_1$ is sufficiently small, then $(x'_n)_{n\geq 1}$ is also a complete AOB. 
\end{proposition}

\begin{proof}
Let $(a_n)_n\in \ell^2$. For the 
first part, it is sufficient to show 
that for sufficiently large $N$, we have 
\begin{equation}
(c_N+ \varepsilon_N-2\sqrt{c_N\varepsilon_N})\displaystyle\sum_{n \geq N} |a_n|^2 \leq \left\|\displaystyle\sum_{n \geq N}a_n x'_n\right\|_\HH^2 \leq (C_n + \varepsilon_N + 2\sqrt{C_N\varepsilon_N}) \displaystyle\sum_{n \geq N}|a_n|^2.
\end{equation}
We define $g_N:= \displaystyle\sum_{n \geq N}a_n x_n$ and $g'_N=\displaystyle\sum_{n \geq N}a_n x'_n$ and write

\begin{eqnarray*}
\|g_N -g'_N\|_\HH^2 &=& \left\langle g_N - g'_N, \displaystyle\sum_{n \geq N}a_n (x_n-x'_n) \right\rangle_\HH \\
&=& \displaystyle\sum_{n \geq N}\overline{a_n} \left\langle g_N-g'_N, x_n - x'_n\right\rangle.
\end{eqnarray*}
Hence, using Cauchy--Schwarz inequality, we have
\begin{eqnarray*}
\|g_N -g'_N\|_\HH^2  &\leq& \left(\displaystyle\sum_{n \geq N}|a_n|^2\right)^{1/2} \left(\displaystyle\sum_{n \geq N}|\langle g_N-g'_N, x_n-x'_n\rangle|^2\right)^{1/2}\\
&\leq& \sqrt \varepsilon_N \left(\displaystyle\sum_{n \geq N}|a_n|^2\right)^{1/2} \|g_N-g'_N\|_\HH.
\end{eqnarray*}
Therefore, we obtain $\|g_N-g'_N\|_\HH \leq \sqrt{\varepsilon_N} \|(a_n)_{n \geq N}\|_{\ell^2}.$
We now obtain the desired inequalities as follows:
\begin{eqnarray*}
\left\|\displaystyle\sum_{n \geq N} a_nx'_n\right\|_\HH &\geq& \left\|\displaystyle\sum_{n \geq N}a_nx_n\right\|_\HH - \left\|\displaystyle\sum_{n \geq N}a_n(x_n-x'_n)\right\|_\HH\\
&\geq& \sqrt{c_N} \left(\displaystyle\sum_{n \geq N} |a_n|^2 \right)^{1/2}- \sqrt{\varepsilon_N} \left(\displaystyle\sum_{n \geq N} |a_n|^2\right)^{1/2} \\
&=& (\sqrt{c_N} - \sqrt{\varepsilon_N}) \left(\displaystyle\sum_{n \geq N} |a_n|^2 \right)^{1/2}.
\end{eqnarray*}
And similarly,
\begin{eqnarray*}
\left\|\displaystyle\sum_{n \geq N} a_nx'_n\right\|_\HH \leq (\sqrt{C_N} + \sqrt{\varepsilon_n}) \left(\displaystyle\sum_{n \geq N} |a_n|^2\right)^{1/2}.
\end{eqnarray*}
Hence $(x'_n)_{n\geq 1}$ is an AOS.

Assume, furthermore that $(x_n)_{n\geq 1}$ is a complete AOB. Then, we know that the operator $J_{\XX^*}$, defined by 
$J_{\XX^*}x=(\langle x,x_n \rangle)_{n\geq 1}$, is an isomorphism from $\HH$ onto $\ell^2$. The inequality \eqref{frame} for $N=1$ implies that $\|J_{\XX^*}-J_{{\XX'}^*}\|\leq \sqrt{\varepsilon_1}$. Therefore for $\varepsilon_1$ sufficiently small, the operator $J_{{\XX'}^*}$ is also an isomorphism from $\HH$ onto $\ell^2$, and, hence, $(x'_n)_{n\geq 1}$ is a complete AOB.

\end{proof}

\begin{proposition}
Let $\XX=(x_n)_{n\geq 1}$ be a complete AOB in $\HH$ and $(x'_n)_{n\geq 1}$ be another sequence in $\HH$ satisfying 
$$
\sum_{n\geq 1}\|x_n-x_n'\|_\HH^2 < \|U_{\XX}\|^{-1}.
$$
Then $(x'_n)_{n\geq 1}$ is a complete AOB in $\HH$. 
\end{proposition}

\begin{proof}
Let $x\in\HH$. Then we have 
$$
\sum_{n\geq N}|\langle x,x_n-x_n'\rangle|^2\leq \|x\|_\HH^2 \sum_{n\geq N}\|x_n-x_n'\|_\HH^2=\varepsilon_N \|f\|_\HH^2,
$$
where $\varepsilon_N=\sum_{n\geq N}\|x_n-x_n'\|_\HH^2$. It follows from hypothesis that $\varepsilon_N\to 0$ as $N$ goes to $\infty$. Hence by Proposition~\ref{Prop:stability23}, the sequence $(x'_n)_{n\geq 1}$ is an AOS. It remains to prove that $(x'_n)_{n\geq 1}$ is minimal and complete. For that purpose, define $T:\HH\longrightarrow\HH$ by $T(x_n)=x_n'$, $n\geq 1$, and let $\delta>0$ such that 
$$
\sum_{n\geq 1}\|x_n-x_n'\|_\HH^2 \leq \delta <\|U_{\XX}\|^{-1}.
$$
Then, for every finitely supported sequence $(a_n)_n$ of complex numbers, we have 
\begin{eqnarray*}
\left\|(I-T)\sum_{n\geq 1}a_n x_n \right\|&=&\left\|\sum_{n\geq 1}a_n (x_n-x'_n) \right|\\
&\leq & \left(\sum_{n\geq 1}|a_n|^2\right)^{1/2} \left(\sum_{n\geq 1}\|x_n-x_n'\|^2\right)^{1/2}\\
&\leq & \sqrt{\delta}\left(\sum_{n\geq 1}|a_n|^2\right)^{1/2}\\
&\leq & \delta \|U_{\XX}\| \left\|\sum_{n\geq 1}a_n x_n \right\|.
\end{eqnarray*}
Since $(x_n)_{n\geq 1}$ is a Riesz basis, we get that $I-T$ is bounded and $\|I-T\|\leq \delta \|U_{\XX}\|<1$. Thus $T=I-(I-T)$ is bounded and invertible. In particular, we deduce that $(x'_n)_{n\geq 1}$ is complete and minimal.

\end{proof}

\subsection{De Branges--Rovnyak spaces}
Let $H^\infty$ denote the space of bounded analytic functions on the upper half-plane $\mathbb{C}_+=\{z\in\mathbb{C}:\Im(z)>0\}$ normed by $\|f\|_\infty=\sup_{z\in\mathbb{C}_+}|f(z)|$ and $H^\infty_1=\{g\in H^\infty:\|g\|_\infty\leq 1\}$ is the closed unit ball of $H^\infty$, and for $b\in H^\infty_1$, the de Branges--Rovnyak spaces $\HH(b)$ is the reproducing kernel Hilbert space of analytic functions on $\mathbb{C}_+$ whose kernel is 
$$
k_\lambda^b(z)=\frac{i}{2\pi}\frac{1-\overline{b(\lambda)}b(z)}{z-\overline{\lambda}},\qquad \lambda,z\in \mathbb{C}_+.
$$
By definition, $f(\lambda)=\langle f,k_\lambda^b\rangle_b$ for all $f\in\HH(b)$ and $\lambda\in\mathbb{C}_+$, where $\langle \cdot,\cdot \rangle_b$ represents the inner product in $\HH(b)$. The space $\HH(b)$ can also be defined as the range space $(I-T_bT_b^*)^{1/2}H^2$ equipped with the norm which makes $(I-T_bT_b^*)^{1/2}$ a partial isometry. Here $H^2$ is the Hardy space of $\mathbb C_+$, that is the space of analytic functions $f$ on $\mathbb C_+$ verifying 
$$
\|f\|_2^2=\sup_{y>0}\left(\int_{-\infty}^\infty |f(x+iy)|^2\,dx\right)<\infty,
$$
$T_\varphi$ is the Toeplitz operator on the Hardy space $H^2$ on $\mathbb{C}_+$ with symbol $\varphi\in L^\infty(\mathbb R)$ defined by $T_\varphi(f)=P_+(\varphi f)$, $f\in H^2$, where $P_+$ denotes the orthogonal projection of $L^2(\mathbb R)$ onto $H^2$. 

These spaces (and, more precisely, their general vector-valued version) were introduced by de Branges and Rovnyak
\cite{de1966canonical, de1966square} as universal model spaces for Hilbert space contractions. Thanks to the
pioneer works of Sarason, we know that de Branges--Rovnyak spaces play an important role in numerous questions of complex analysis and
operator theory. The book \cite{sarason1994sub} is the classical reference for $\HH(b)$ spaces. See also the forthcoming monography \cite{fricain2014theory}. 

In the special case where $b=\Theta$ is an inner
function (that is, $|\Theta|=1$ a.e. on $\mathbb{R}$), the operator
$(Id-T_\Theta T_{\Theta}^*)^{1/2}$ is an orthogonal projection and
$\mathcal{H}(\Theta)$ becomes a closed (ordinary) subspace
of $H^2$ which coincides with the so-called model subspace
\[
K_\Theta=H^2\ominus \Theta
H^2=H^2\cap \Theta \, \overline{H^2}.
\]
For the model space theory see \cite{nikolski2001operators}.

It turns out that the boundary behavior of functions in $\HH(b)$ is controlled by the boundary behavior of the function $b$ itself. More precisely, let $b=B I_\mu
O_b$ be the canonical factorization of $b$, where
$$
B(z) = \prod_n e^{i\alpha_n}\frac{z-z_n}{z-\overline{z_n}}
$$
is a Blaschke product, the singular inner function $I_\mu$ is given
by
$$
I_\mu(z) = \exp\left(iaz - \frac{i}{\pi}\int_\mathbb{R}
\bigg(\frac{1}{z-t} + \frac {t}{t^2+1}\bigg)\,d\mu(t)\right)
$$
with a positive measure $\mu$ on $\mathbb{R}$ singular with respect to Lebesgue measure $dt$ such that $\int_{\mathbb R}(1+t^2)^{-1}\,d\mu(t)<\infty$ and $a\ge 0$,
and $O_b$ is the outer function
$$
O_b(z) = \exp\left(\frac{i}{\pi}\int_\mathbb{R} \bigg(\frac{1}{z-t} + \frac
{t}{t^2+1}\bigg) \log|b(t)|\,dt\right).
$$
For $x_0\in\mathbb R$ and $\ell\geq 1$, let 
$$
S_\ell(x_0):=\sum_{n=1}^\infty \frac{\Im(z_n)}{|x_0-z_n|^\ell}+\int_{\mathbb R}\frac{d\mu(t)}{|x_0-t|^\ell}+\int_{\mathbb R}\frac{|\log|b(t)||}{|x_0-t|^\ell}\,dt,
$$
and $E_\ell(b)=\{x_0\in\mathbb{R}:S_\ell(x_0)<\infty\}$. The set $E_\ell(b)$ is related to nontangential boundary limits of functions (and their derivatives) in $\HH(b)$. More precisely, if $S_2(x_0)<\infty$, then it is proved in \cite{fricain2008boundary} that for each $f\in\HH(b)$,
the nontangential limit
$$
f(x_0)=\lim_{\substack{\,\,z  \longrightarrow  x_0\\ \sphericalangle}}  f(z)
$$
exists, the function
$$
k_{x_0}^b(z)=\frac{i}{2\pi}\,\frac{1-\overline{b(x_0)}b(z)}{z-x_0},\qquad z\in\mathbb C_+,
$$
belongs to $\HH(b)$ and $\langle f,k_{x_0}^b\rangle_b=f(x_0)$, $f\in\HH(b)$. In that case, we also have $\|k_{x_0}^b\|_b^2=S_2(x_0)=|b'(x_0)|$.  Moreover if $S_4(x_0)<\infty$, for every function $f\in\HH(b)$, $f(z)$ and $f'(z)$ have finite limits as $z$ tends nontangentially to $x_0$.  In \cite{baranov2010weighted}, some Bernstein's type inequality is proved in the $\HH(b)$ space. To state this inequality, we need to introduce the following kernel. For $z_0\in\mathbb C_+\cup E_4(b)$, we define 
$$
\mathfrak{K}_{z_0}^b(t)=\overline{b(z_0)}\,\frac{2-\overline{b(z_0)}b(t)}{(t-\overline{z_0})^2}.
$$
It is not difficult to see that 
$\rho^{1/q}\mathfrak{K}_{z_0}^b \in L^q(\mathbb R)$ if and only if 
$$
\int_{\mathbb R}\frac{|\log|b(t)||}{|t-z_0|^{2q}}\,dt<\infty,
$$
where $\rho(t)=1-|b(t)|^2$, $t\in\mathbb R$. Now, for $1<p\leq 2$ and $q$ its conjugate exponent, we define 
$$
w_p(z):=\min\left\{\, \|{(k_z^b)}^2\|_q^{-p/(p+1)}, \,\,
\|\rho^{1/q}\mathfrak{K}_z^b\|_q^{-p/(p+1)} \,\right\},\qquad z\in\overline{\mathbb C_+},
$$
where $\|\cdot\|_q$ denotes the $L^q(\mathbb R)$-norm with respect to Lebesgue measure $dt$ on $\mathbb R$. 

We assume $w_{p,n}(x)=0$, whenever $x\in \mathbb{R}$ and at least one of the functions $(k_x^b)^2$ or $\rho^{1/q}
\mathfrak{K}_x^b$ is not in $L^q(\mathbb{R})$. Note that if $f\in\HH(b)$ and $1<p\leq 2$, then $f'w_p$  is well defined on $\mathbb R$. Indeed, if $S_4(x)<\infty$, then $f'(x)$ and $w_p(x)$ are finite. If $S_4(x)=\infty$, then as shown, in \cite{baranov2010weighted,fricain2008boundary},  $\|(k_x^b)^2\|_q=\infty$, which, by definition, implies that $w_p(x)=0$, and thus we may assume that $(f'w_p)(x)=0$. Moreover, note that in the inner case, we have $\rho(t)=0$ for a.e. $t\in\mathbb R$, and the second term in the definition of the weight $w_p$ disappears. We will need two useful estimates for the weight $w_p$. The first one, proved in~\cite[Lemma 3.5]{baranov2010weighted},  is valid for every function $b\in H^\infty_1$: there is a constant
$C=C(p)>0$ such that
\begin{equation}\label{eq:estimate1-kernel-Bernstein}
w_p(z)\geq C \frac{\Im z}{(1-|b(z)|)^{\frac{p}{q(p+1)}}},\qquad (z\in\mathbb{C}_+).
\end{equation}
The second one, proved in \cite{baranov2005bernstein} and valid when $b=\Theta$ is an inner function, says that there is two constants $C_1,C_2>0$ such that 
\begin{equation}\label{eq:estimate2-kernel-Bernstein}
C_1 v_0(x)\leq w_p(x) \leq C_2 |\Theta'(x)|^{-1},\qquad (x\in\R),
\end{equation}
where $v_0(x)=\min(d_0(x),|\Theta'(x)|^{-1})$, $d_0(x)={\dist}(x,\sigma(\Theta))$ and $\sigma(\Theta)$ is the spectrum of the inner function $\Theta$ defined as the set of all 
$\zeta\in\overline{\mathbb C_+}\cup\{\infty\}$ such that $\liminf_{z\to\zeta\,z\in\mathbb C_+}|\Theta(z)|=0$. Remind that every function $f\in K_\Theta$ has an analytic continuation through $\mathbb{R}\setminus\sigma(\Theta)$. Moreover, the quantity $v_0$ has a simple geometrical meaning related to the sublevel sets $\Omega(\Theta,\delta)=\{z\in\mathbb{C}_+:|\Theta(z)|\leq \delta\}$. Namely, $v_0(x)\asymp {\dist}(x,\Omega(\Theta,\delta))$ with the constants depending only on $\delta\in (0,1)$. 

We also remind that a Borel measure $\mu$ in the closed upper half-plane $\overline{\mathbb{C}_+}$ is said to be a \emph{Carleson measure} if there is a
constant $C>0$ such that
\begin{equation}\label{carl} 
\mu(\, S(x,h) \,)\le C \, h,
\end{equation}
for all squares $S(x,h)=[x,x+h]\times[0,h]$, $x\in\mathbb{R}$, $h>0$,
with the lower side on the real axis. We denote the class of Carleson measures by $\mathcal{C}$ and the best constant satisfying \eqref{carl} is called the \emph{Carleson constant} of $\mu$ and is denoted by $C_\mu$. Recall that, according to a classical
theorem of Carleson, $\mu \in \mathcal{C}$ if and only if
$H^p\subset L^p(\mu)$ for some (all) $p>0$.
In \cite{baranov2010weighted}, it is proved that if $\mu\in\mathcal{C}$, $1<p<2$, then there exists a constant $K=K(\mu,p)>0$ such that 
\begin{equation}\label{eq:mainineq}
\|f'w_p\|_{L^2(\mu)} \le K\|f\|_b,
\qquad f\in \mathcal{H}(b).
\end{equation}
In other words, the map $f\longmapsto f'w_p$ is a bounded operator from $\HH(b)$ into $L^2(\mu)$. In the case $p=2$, then this map is of weak type $(2,2)$ as an operator from $\mathcal{H}(b)$ to $L^2(\mu)$.

\section{Some stability results}
This section contains results about the stability of AOBs under certain perturbations. We will often use techniques developed by Baranov \cite{baranov2005stability} concerning the stability problem for Riesz basis in $K_\Theta$.  

For $\lambda\in\mathbb{C}_+\cup E_2(b)$, we denote by $\kappa_\lambda^b$ the
normalized reproducing kernel at the point $\lambda$, that is,
$\kappa_\lambda^b= k_\lambda^b/\|k_\lambda^b\|_b$. Let
$(\kappa_{\lambda_n}^b)_{n\geq 1}$ be an AOS in $\mathcal{H}(b)$, let $G=\bigcup_n G_n\subset
\overline{\mathbb{C}_+}$ satisfy the following properties:
\begin{enumerate}
\item[$\mathrm{(i)}$] $\lambda_n\in G_n$; 
\item[$\mathrm{(iii)}$] for every $z_n\in G_n$, we have 
\[
\lim_{n\to \infty} \frac{\|k_{z_n}^b\|_b}{\|k_{\lambda_n}^b\|_b}=1 ;
\]
\item[$\mathrm{(ii)}$] for every $z_n\in G_n$, the measure
$\nu=\sum_n \delta_{[\lambda_n,z_n]}$
is a Carleson measure and, moreover, the Carleson constants $C_\nu$ of such
measures (see (\ref{carl})) are uniformly bounded with respect to $z_n$.
Here $[\lambda_n,z_n]$ is the straight line interval with the endpoints
$\lambda_n$ and $z_n$, and $\delta_{[\lambda_n,z_n]}$ is the Lebesgue measure
on the interval.
\end{enumerate}
For 
$(\lambda_n)_{n\geq 1}\subset\mathbb C_+$, we show there always exist non-trivial sets $G_n$ satisfying $(i)$, $(ii)$ and $(iii)$. More
precisely, we can take
\[
G_n:=\{z\in\mathbb{C}_+:|z-\lambda_n|<\varepsilon_n\Im\lambda_n\},
\]
where $(\varepsilon_n)_n$ is any sequence of positive numbers tending to $0$. We first begin with a technical lemma. 
\bl\label{Lemma:stability-condition-ii}
Let $b\in H^\infty_1$, let $(\varepsilon_n)_n$ be a sequence of positive numbers such that $\varepsilon_n\to 0$, as $n\to\infty$, and let $(\lambda_n)_n$ and $(\mu_n)_n$ be sequences in $\mathbb{C}_+$ satisfying 
\begin{equation}\label{eq:stability-result-ensemble-Gn}
|\lambda_n-\mu_n|\leq \varepsilon_n \Im\lambda_n.
\end{equation}
Then 
$$
\lim_{n\to \infty} \frac{\|k_{\mu_n}^b\|_b}{\|k_{\lambda_n}^b\|_b}=1.
$$
\el

\begin{proof}
We easily check from \eqref{eq:stability-result-ensemble-Gn} that 
\begin{equation}\label{eq2:stability-result-ensemble-Gn}
1-\varepsilon_n\leq \frac{\Im\mu_n}{\Im\lambda_n}\leq 1+\varepsilon_n,\qquad n\geq 1.
\end{equation}
Since 
$$
\|k_z^b\|_b^2=\frac{1-|b(z)|^2}{4\pi\Im z},
$$
it is sufficient to prove that 
\begin{equation}\label{eq3:stability-result-ensemble-Gn}
\frac{1-\varepsilon_n}{1+\varepsilon}\leq \frac{1-|b(\lambda_n)|}{1-|b(\mu_n)|}\leq \frac{1+\varepsilon_n}{1-\varepsilon_n}.
\end{equation}
Using Schwarz--Pick inequality, we have 
$$
\left|\frac{b(\lambda_n)-b(\mu_n)}{1-\overline{b(\lambda_n)}b(\mu_n)}\right|\leq \left|\frac{\lambda_n-\mu_n}{\lambda_n-\overline{\mu_n}}\right|\leq \frac{|\lambda_n-\mu_n|}{\Im\lambda_n}\leq\varepsilon_n,
$$
and \eqref{eq3:stability-result-ensemble-Gn} follows from 
\cite[Lemma 7]{havin1974free} which says that if $\lambda,\mu\in\mathbb{D}$ and satisfies 
$$
\left|\frac{\lambda-\mu}{1-\overline\lambda \mu}\right|\leq \varepsilon, 
$$
then 
$$
\frac{1-\varepsilon}{1+\varepsilon}\leq \frac{1-|\lambda|}{1-|\mu|}\leq \frac{1+\varepsilon}{1-\varepsilon}.
$$

\end{proof}

\begin{corollary}\label{Cor:existence-set-Gn}
Let $b\in H^\infty_1$, let $(\lambda_n)_{n\geq 1}\subset\mathbb{C}_+$ such that $(\kappa_{\lambda_n}^b)_n$ is an AOS. Let $(\varepsilon_n)_{n\geq 1}$ be a sequence of positive numbers such that $\varepsilon_n\to 0$ as $n\to\infty$. Define 
$$
G_n:=\{z\in\mathbb{C}_+:|z-\lambda_n|<\varepsilon_n\Im\lambda_n\},\qquad n\geq 1.
$$
Then the set $G_n$ satisfies $(i),(ii)$ and $(iii)$. 
\end{corollary}

\begin{proof}
It is obvious that $G_n$ satisfies $(i)$ and condition $(ii)$ follows from Lemma~\ref{Lemma:stability-condition-ii}. According to 
Proposition~\ref{cns-general-AOB}, there exists a constant $C>0$ such that for every $n\geq 1$, we have 
$$
\sum_{p\geq 1}|\Gamma_{n,p}|^2\leq C,
$$
where $\Gamma_{n,p}=\langle \kappa_{\lambda_n}^b, \kappa_{\lambda_p}^b\rangle_b$. 
Since
$$
|\Gamma_{n,p}|^2=\frac{16\pi^2\Im\lambda_n \Im\lambda_p}{|\lambda_p-\overline{\lambda_n}|^2}\frac{|1-\overline{b(\lambda_n)}b(\lambda_p|^2}{(1-|b(\lambda_n|^2)(1-|b(\lambda_p)|^2}\geq \frac{\Im\lambda_n \Im\lambda_p}{|\lambda_p-\overline{\lambda_n}|^2},
$$
we obtain 
$$
\sum_{p\geq 1} \frac{\Im\lambda_n \Im\lambda_p}{|\lambda_p-\overline{\lambda_n}|^2} \leq C.
$$
It is known (see for instance \cite[Lecture VII]{nikolski1986treatise}) that this condition implies that the measure $\nu=\sum_n \Im\lambda_n \delta_{\lambda_n}$ is a Carleson measure. Therefore, the set $G_n$ also satisfies $(iii)$. 

\end{proof}

Note that in \cite{baranov2005stability,baranov2010weighted}, similar sets were considered in connection with stability of Riesz property. In that situation, condition (ii) can be replaced by the weaker condition that there exists two positive constants $c,C>0$ such that 
$$
c\leq \frac{\|k^b_{z_n}\|_b}{\|k^b_{\lambda_n}\|_b}\leq C,\qquad  z_n\in G_n, n\geq 1,
$$
and the set $G_n$ can be taken as
\[
G_n:=\{z\in\mathbb{C}_+:|z-\lambda_n|<r\Im\lambda_n\},
\]
for sufficiently small $r>0$.

\begin{theorem}\label{prop1}
Let $b\in H^\infty_1$, let $(\lambda_n)_n\subset\mathbb C_+\cup E_2(b)$ be such that $(\kappa_{\lambda_n}^b)_{n\geq 1}$ is an AOS in $\HH(b)$, let $1<p<2$, let $G=\cup_{n\geq 1}G_n$ satisfy $(i), (ii)$ and $(iii)$, and let $\mu_n\in G_n$, $n\geq 1$. Suppose there exists $N_0\in\mathbb N$ such that for all $N\geq N_0$ there is $\varepsilon_N>0$ verifying 
\begin{equation}\label{eq0:thm-stability-Baranov-type}
\sup_{n \geq N} \frac{1}{\|k^b_{\lambda_n}\|_b^2} \int_{[\lambda_n, \mu_n]} w_p^{-2}(z) |dz| \leq  \varepsilon_N,
\end{equation}
and $\lim_{N\to\infty}\varepsilon_N=0$. Then 
the sequence $(\kappa_{\mu_n}^b)_{n\geq 1}$ is an AOS in $\HH(b)$. Moreover, if $(\kappa_{\lambda_n}^b)_{n\geq 1}$ is a complete AOB in $\HH(b)$ and if we can take $N_0=1$ and $\varepsilon_1$ sufficiently small, then $(\kappa_{\mu_n}^b)_{n\geq 1}$ is also a complete AOB in $\HH(b)$.
\end{theorem}

\begin{proof}
Let $h_{n}^b=\frac{k_{\mu_n}^b}{\|k_{\lambda_n}^b\|_b}$, $n\geq 1$. Since $\frac{\|k_{\mu_n}^b\|_b}{\|k_{\lambda_n}^b\|_b}\to 1$ as $n\to\infty$, we easily see that $(\kappa_{\mu_n}^b)_{n\geq 1}$ is an AOS if and only if $(h_n^b)_{n\geq 1}$ is an AOS. In view of Proposition~\ref{Prop:stability23}, it is then sufficient to check the estimate 
\begin{equation}\label{eq:thm-stability-Baranov-type}
\sum_{n\geq N}|\langle f,\kappa_{\lambda_n}^b-h_n^b\rangle|^2\leq \varepsilon_N \|f\|_b^2,\qquad f\in\HH(b).
\end{equation}
It follows from \eqref{eq0:thm-stability-Baranov-type} and \cite[Corollary 5.4]{baranov2010weighted} that any $f\in\HH(b)$ is differentiable in $]\lambda_n,\mu_n[$ and, the set of all functions in $\HH(b)$ which are continuous on $[\lambda_n,\mu_n]$ is dense in $\HH(b)$. Therefore, it is sufficient to prove \eqref{eq:thm-stability-Baranov-type} for functions $f\in\HH(b)$ continuous on $[\lambda_n,\mu_n]$. Then 
$$
|\langle f,\kappa_{\lambda_n}^b-h_n^b\rangle|^2=\frac{|f(\lambda_n)-f(\mu_n)|^2}{\|k_{\lambda_n}^b\|_b^2}=\frac{1}{\|k_{\lambda_n}^b\|_b^2}\left|\int_{[\lambda_n,\mu_n]}f'(z)\,dz\right|^2.
$$
By the Cauchy--Schwartz inequality and \eqref{eq0:thm-stability-Baranov-type}, we get for $n\geq N$, 
$$
|\langle f,\kappa_{\lambda_n}^b-h_n^b\rangle|^2\leq \varepsilon_N \int_{[\lambda_n,\mu_n]} |f'(z)w_p(z)|^2\,|dz|.
$$
It follows from assumption $(iii)$ that $\nu=\sum_{n}\delta_{[\lambda_n,\mu_n]}$ is a Carleson measure with a constant $C_\nu$ which does not exceed some absolute constant depending only on $G$. Hence according to \eqref{eq:mainineq}, we have
$$
\sum_{n\geq N}|\langle f,\kappa_{\lambda_n}^b-h_n^b\rangle|^2\leq \varepsilon_N \sum_{n\geq N} \int_{[\lambda_n,\mu_n]} |f'(z)w_p(z)|^2\,|dz|\leq \varepsilon_N \|f'w_p\|^2_{L^2(\nu)}\leq K \varepsilon_N \|f\|_b^2.
$$
Then, since $\varepsilon_N\to 0$ as $N\to \infty$, Proposition~\ref{Prop:stability23} implies that $(h_n^b)_{n\geq 1}$ is an AOS, and so is $(\kappa_{\mu_n}^b)_{n\geq 1}$. The second part of the result for complete AOB follows also from Proposition~\ref{Prop:stability23}. 

\end{proof}

\begin{remark}\rm{
If $(\kappa_{\lambda_n}^b)_{n\geq 1}$ is a complete AOB in $\HH(b)$ and $(\lambda_n)_{n\geq 1}\subset\mathbb{C}_+$, then it is sufficient to have \eqref{eq0:thm-stability-Baranov-type} for $N$ large enough to get that $(\kappa_{\mu_n}^b)_{n\geq 1}$ is a complete AOB in $\HH(b)$. Indeed, combine Theorem~\ref{prop1} applied to the sequence 
$$
\gamma_n=\begin{cases}
\lambda_n & \hbox{if }n\leq N\\
\mu_n & \hbox{if }n>N
\end{cases},
$$
with part (a) of the following lemma which shows that we can replace a finite number of terms keeping the minimality and completeness.
}
\end{remark}

\bl\label{Lem:add-minimality}
Let $b\in H^\infty_1$ and $\Lambda=(\lambda_n)_{n\geq 1}\subset\mathbb C_+$. 
\begin{enumerate}
\item[(a)] Assume that $(k_{\lambda_n}^b)_{n\geq 1}$ is a minimal and complete sequence in $\HH(b)$. Then, for every $\mu\in\mathbb{C}_+\setminus\Lambda$, the system $\{k_{\lambda_n}^b\}_{n\geq 2}\cup\{k_\mu^b\}$ is still minimal and complete in $\HH(b)$. 
\item[(b)] Assume that $(k_{\lambda_n}^b)_{n\geq 1}$ is not complete in $\HH(b)$. Then for every $\mu\in\mathbb C_+\setminus\Lambda$, the system $\{k_{\lambda_n}^b\}_{n\geq 1}\cup\{k_\mu^b\}$ is minimal. 
\end{enumerate}
\el

This result is proved in \cite{hruvsvcev1981unconditional} for the inner case. The general version is proved similarly; see \cite[Lemma 31.2]{fricain2014theory}. We also use later a version of this result for real frequencies. We do not know if it true in general but we prove it when $b=\Theta$ is an inner function. The proof uses the following key lemma. 
\bl\label{Lem-key-minimality-stability}
Let $\Theta$ be an inner function, $x_0\in\mathbb R\setminus\sigma(\Theta)$, and $f\in K_\Theta$ such that $f(x_0)=0$. Then there exists a Blaschke factor $J$ such that $\{J=-1\}=\{x_0\}$ and $f/(1+J)\in K_\Theta$.
\el
\begin{proof}
Fix any $a>0$ and define $\gamma=x_0+ia\in\mathbb{C}_+$. Then $J$ will be the Blaschke factor $b_\gamma$, that is 
$$
J(z)=b_\gamma(z)=\frac{z-\gamma}{z-\overline\gamma}.
$$
An easy computation shows that 
$$
1+J(z)=\frac{2(z-x_0)}{z-\overline\gamma},
$$
and in particular, we get that $\{J=-1\}=\{x_0\}$. To check that $f/(1+J)\in K_\Theta$, first note that 
$$
\frac{f(z)}{1+J(z)}=\frac{1}{2}\left(f(z)+ia\frac{f(z)}{z-x_0}\right).
$$
Since $x_0\in\mathbb{R}\setminus\sigma(\Theta)$, we know that $f$ is analytic in a neighbourhood $V_{x_0}$ of $x_0$ and thus we have 
$$ 
|f(z)|\leq C|z-x_0|,\qquad z\in V_{x_0}.
$$
Hence $f/(z-x_0)\in L^2(\mathbb R)\cap\mathcal N^+=H^2$, where $\mathcal N^+$ is the Smirnov class. We deduce that $f/(1+J)\in H^2$. It remains to notice that 
$$
\frac{\Theta \bar f}{1+\bar J}=\frac{J\Theta\bar f}{1+J},
$$
and since $f\in K_\Theta$, we have $\Theta \bar f\in H^2$. Thus $\Theta \bar f/(1+\bar J)\in L^2(\mathbb R)\cap\mathcal N^+=H^2$ and then $f/(1+J)\in H^2\cap \Theta \overline{H^2}=K_\Theta$.  

\end{proof}

\bl\label{Lem2:add-minimality}
Let $\Theta$ be an inner function and $(t_n)_{n\geq 1}\subset\mathbb{R}$. 
\begin{enumerate}
\item[(a)] Assume that $t_1\not\in\sigma(\Theta)$ and $(k_{t_n}^\Theta)_{n\geq 1}$ is a minimal and complete sequence in $K_\Theta$. Then, for every $t\in\mathbb{R}\setminus\sigma(\Theta)$ and $t\not=t_n$, $n\geq 1$, the system $\{k_{t_n}^\Theta\}_{n\geq 2}\cup\{k_t^\Theta\}$ is still minimal and complete in $K_\Theta$. 
\item[(b)] Assume that $t_n\not\in\sigma(\Theta)$, $n\geq 1$, and $(k_{t_n}^\Theta)_{n\geq 1}$ is not complete in $K_\Theta$.
Then for every $t\in\mathbb R\setminus\sigma(\Theta)$ and $t\not=t_n$, $n\geq 1$, the system $\{k_{t_n}^\Theta\}_{n\geq 1}\cup\{k_t^\Theta\}$ is minimal. 
\end{enumerate}
\el
\begin{proof}
(a): let us first prove that the system $\{k_{t_n}^\Theta\}_{n\geq 2}\cup\{k_t^\Theta\}$ is complete. So let $f\in K_\Theta$ $f(t_n)=0$, $n\geq 2$ and $f(t)=0$. According to Lemma~\ref{Lem-key-minimality-stability}, there is an inner function $J$ such that $\{J=-1\}=\{t\}$ and $f/(1+J)\in K_\Theta$. Consider now 
$$
g=\frac{J-J(t_1)}{1+J}f=f-(J(t_1)+1)\frac{f}{1+J}.
$$
The function $g$ belongs to $K_\Theta$ and it vanishes at $t_n$, $n\geq 1$. Hence, the completeness of $(k_{t_n}^\Theta)_{n\geq 1}$ implies that $g\equiv 0$ and thus $f\equiv 0$. That proves the completeness of $\{k_{t_n}^\Theta\}_{n\geq 2}\cup\{k_t^\Theta\}$. As far as concerned the minimality, note that for every $n\geq 1$, there exists a function $f_n\in K_\Theta$ such that $f_n(t_\ell)=\delta_{n,l}$. By completeness of $\{k_{t_n}^\Theta\}_{n\geq 2}\cup\{k_t^\Theta\}$, we necessarily have $f(t)\neq 0$. Then that proves that $k_t^\Theta\not\in\span(k_{t_n}^\Theta:n\geq 2)$. Fix now $n\geq 2$. Using once more time Lemma~\ref{Lem-key-minimality-stability}, there is an inner function $J_1$ such that $\{J_1=-1\}=\{t_1\}$ and $f_n/(1+J_1)\in K_\Theta$. Now consider the function $g_n=((J_1-J_1(t))f)/(1+J_1)$. It is clear that $g_n\in K_\Theta$. Moreover, we have $g_n(t)=0$, $g_n(t_\ell)=0$, $\ell\not=n$ and $g_n(t_n)=(J_1(t_n)-J_1(t))/(1+J_1(t_n))\neq 0$ (since $J_1$ is a Blaschke factor and thus is one-to-one). Hence we get that $k_{t_n}^\Theta\not\in\span(\{k_{t_\ell}^\Theta\}_{\ell\geq 2,\ell\not=n}\cup\{k_t^\Theta\})$. That proves the minimality of $\{k_{t_n}^\Theta\}_{n\geq 2}\cup\{k_t^\Theta\}$.

(b): since $(k_{t_n}^\Theta)_{n\geq 1}$ is not complete in $K_\Theta$, there exists a function $f\in K_\Theta$, $f\not\equiv 0$, such that $f(t_n)=0$, $n\geq 1$. Fix $n\geq 1$. By Lemma~\ref{Lem-key-minimality-stability}, there is a Blaschke factor $J_n$ such that $\{J_n=-1\}=\{t_n\}$ and $f/(1+J_n)\in K_\Theta$. Consider now the function $f_n=((J_n-J_n(t))f)/(1+J_n)$. Then $f_n\in K_\Theta$ and we have $f_n(t)=0$, $f_n(t_\ell)=0$, $\ell\not=n$. Dividing once more time by $1+J_n$ if necessary, we can assume that $f_n(t_n)\neq 0$. Hence we deduce that $k_{t_n}^\Theta\not\in\span(\{k_{t_\ell}^\Theta\}_{\ell\geq 1,\ell\neq n}\cup\{k_t^\Theta\})$. On the other hand, if $f(t)\neq 0$, then we immediately get that $k_t^\Theta\not\in\span(k_{t_n}^\Theta:n\geq 1)$. If $f(t)=0$, then we can use once more time Lemma~\ref{Lem-key-minimality-stability} to drop of that extra zero. That proves the minimality of $\{k_{t_n}^\Theta\}_{n\geq 1}\cup\{k_t^\Theta\}$.

\end{proof}

Let $\Theta$ be an inner function, $(\lambda_n)_n\subset\mathbb{C}_+$ satisfying $\sup_{n\geq 1}|\Theta(\lambda_n)|<1$. It is proved in \cite{chalendar2003functional} that if $(\kappa_{\lambda_n}^\Theta)_{n\geq 1}$ is an AOS, there exists $\varepsilon>0$ such that $(\kappa_{\mu_n}^b)_{n\geq 1}$ is an AOS for any sequence 
$(\mu_n)_{n\geq 1}\in\mathbb{C}_+$ satisfying 
$$
\left|\frac{\lambda_n-\mu_n}{\lambda_n-\overline{\mu_n}}\right|\leq \varepsilon.
$$
It is easy to see that this can be generalized to the general case when the inner function $\Theta$ is replaced by a function $b\in H^\infty_1$; see \cite{fricain2014theory}. Without the hypothesis that $\sup_{n\geq 1}|b(\lambda_n)|<1$, we obtain the following stability result concerning pseudo-hyperbolic perturbations.

\begin{corollary}\label{cor1-perturbation}
Let $b\in H^\infty_1$, $(\lambda_n)_{n\geq 1}\subset\mathbb{C}_+$ such that $(\kappa_{\lambda_n}^b)_{n\geq 1}$ is an AOS in $\HH(b)$. Let $\gamma>1/3$ and $(\varepsilon_n)_{n\geq 1}$ be a sequence of positive numbers tending to $0$. For every sequence $(\mu_n)_{n\geq 1}$ satisfying
\begin{equation}\label{eq:stability-distance-hyp}
\left|\frac{\lambda_n-\mu_n}{\lambda_n-\overline{\mu_n}}\right|\leq \varepsilon_n (1-|b(\lambda_n)|)^\gamma,\qquad n\geq 1,
\end{equation}
the sequence $(\kappa_{\mu_n}^b)_{n\geq 1}$ is an AOS. Moreover, if $(\kappa_{\lambda_n}^b)_{n\geq 1}$ is a complete AOB in $\HH(b)$ then $(\kappa_{\mu_n}^b)_{n\geq 1}$ is also a complete AOB in $\HH(b)$.
\end{corollary}

\begin{proof}
According to Corollary~\ref{Cor:existence-set-Gn}, the sets
$G_n=\{z\in\mathbb{C}_+:|z-\lambda_n|\leq \varepsilon_n\Im\lambda_n\}$
satisfy the conditions $(i)$, $(ii)$ and $(iii)$. Let
$(\mu_n)_{n\geq 1}$ satisfy \eqref{eq:stability-distance-hyp}. Then, we have
\begin{eqnarray}\label{eq4:stabilite}
|\lambda_n-\mu_n|\leq
\varepsilon_n (1-|b(\lambda_n)|)^\gamma\Im\lambda_n\leq \varepsilon_n \Im\lambda_n .
\end{eqnarray}
Therefore, $\mu_n\in G_n$.
Without loss of generality, we can assume that $\gamma<1$, and since
$\gamma>1/3$, there exists $1<p<2$ such that
$2\frac{p-1}{p+1}=1-\gamma$.
Let $q$ be the conjugate exponent of $p$ and note that
$\frac{2p}{q(p+1)}=1-\gamma$. Using \eqref{eq:estimate1-kernel-Bernstein}, \eqref{eq2:stability-result-ensemble-Gn} and \eqref{eq3:stability-result-ensemble-Gn} we have
\[
w_p^{-2}(z)\leq C_1\frac{(1-|b(\lambda_n)|)^{1-\gamma}}{(\Im \lambda_n)^{2}}
\]
for $z\in [\lambda_n,\mu_n]$. Hence,
\[
\frac{1}{\|k_{\lambda_n}^b\|_b^2}\int_{[\lambda_n,\mu_n]}w_p(z)^{-2}|dz|\leq
C_2\frac{\Im\lambda_n}{1-|b(\lambda_n)|}|\lambda_n-\mu_n|\frac{(1-|b(\lambda_n)|)^{1-\gamma}}
{(\Im\lambda_n)^2}
\]
and using (\ref{eq4:stabilite}), we obtain
\[
\frac{1}{\|k_{\lambda_n}^b\|_b^2}\int_{[\lambda_n,\mu_n]}w_p(z)^{-2}|dz|\leq
C_3\varepsilon_n.
\]
The conclusion for AOS now follows from Theorem \ref{prop1}.

For complete AOB, we argue as follows. Let 
$$
\gamma_n=\begin{cases}
\lambda_n & \hbox{if }n<N_0\\
\mu_n & \hbox{if }n\geq N_0
\end{cases},
$$
where $N_0$ will be choosen later. Then, the preceeding computations show that 
$$
\sup_{n\geq 1}\frac{1}{\|k_{\lambda_n}^b\|_b^2}\int_{[\lambda_n,\mu_n]}w_p(z)^{-2}|dz|\leq
C_3\sup_{n\geq N_0}\varepsilon_n.
$$
Hence we can find $N_0$ such that $C_3\sup_{n\geq N_0}\varepsilon_n$ is sufficiently small so that according to Theorem~\ref{prop1}, we will get that $(\kappa_{\gamma_n}^b)_{n\geq 1}$ is a complete AOB in $\HH(b)$. Then, we can apply Lemma~\ref{Lem:add-minimality} to get that $(\kappa_{\mu_n}^b)_{n\geq 1}$ is a complete and minimal sequence in $\HH(b)$. Since it is also an AOS, then it is a complete AOB. 

\end{proof}

\begin{remark}\rm{
Note that in the case when $\lim_{n\to\infty}|b(\lambda_n)|=1$, the condition \eqref{eq:stability-distance-hyp} can be replaced by the existence of a constant $C>0$ such that 
$$
\left|\frac{\lambda_n-\mu_n}{\lambda_n-\overline{\mu_n}}\right|\leq C (1-|b(\lambda_n)|)^\gamma,\qquad n\geq 1.
$$
Indeed, it is sufficient to take $\gamma>\gamma_0>1/3$ and note that 
$$
C(1-|b(\lambda_n)|)^\gamma=\varepsilon_n (1-|b(\lambda_n)|)^{\gamma_0},
$$
with $\varepsilon_n=C(1-|b(\lambda_n)|)^{\gamma-\gamma_0}\to 0$ as $n\to\infty$.}
 
\end{remark}
In the inner case, we can also give a stability result when the sequences $(\lambda_n)_n$ and $(\mu_n)_n$ are on the real line. We first need a result on the construction of sets $G_n$. 

\bl\label{Lem:technical-existence-set-Gn-real-spectrum}
Let $\Theta$ be an inner function, let $t_n\in\R$ such that $(\kappa_{t_n}^\Theta)_{n\geq 1}$ is a Riesz basis of $K_\Theta$. Let $(\varepsilon_n)_{n\geq 1}$ be a sequence of positive numbers tending to $0$. Define 
\begin{equation}\label{eq:definition-Gn-real-spectrum}
G_n=\{t\in\R:|t-t_n|\leq \varepsilon_n v_0(t_n)\},
\end{equation}
where $v_0(t)=\min(d_0(t),|\Theta'(t)|^{-1})$ and $d_0(t)={\dist}(t,\sigma(\Theta))$. Then the set $G_n$ satisfies $(i), (ii)$ and $(iii)$. 

\el

\begin{proof}
Consider the nontrivial case when $v_0(t_n)>0$. In particular, we have 
$$
|t-t_n|\leq \varepsilon_n d_0(t_n),\qquad t\in G_n.
$$
Hence
\begin{equation}\label{eq:Lem:technical-existence-set-Gn-real-spectrum}
(1-\varepsilon_n) d_0(t_n)\leq d_0(t)\leq (1+\varepsilon_n) d_0(t_n),\qquad t\in G_n. 
\end{equation}
Now remember that when $t\in\R$, $k_t^\Theta\in K_\Theta$ if and only if 
$$
|\Theta'(t)|=a+\sum_{\ell=1}^\infty \frac{2\Im z_\ell}{|t-z_\ell|^2}+\int_\R \frac{d\sigma(x)}{|t-x|^2}<\infty,
$$
and in that case 
\begin{equation}\label{eq2:Lem:technical-existence-set-Gn-real-spectrum}
\|k_t^\Theta\|_2^2=|\Theta'(t)|.
\end{equation}
Here $(z_\ell)_\ell$ is the sequence of zeros of $\Theta$ and $\sigma$ is its associated singular measure. Using \eqref{eq:Lem:technical-existence-set-Gn-real-spectrum}, it is not difficult to check that for every $\ell\geq 1$ and $t\in G_n$, 
$$
1-\varepsilon_n\leq \frac{|t-z_\ell|}{|t_n-z_\ell|} \leq 1+\varepsilon_n,
$$
and for any $x\in\supp\sigma$, 
$$
1-\varepsilon_n\leq \frac{|t-x|}{|t_n-x|} \leq 1+\varepsilon_n.
$$
Hence 
\begin{equation}\label{eq:derivee-theta-perturbation}
\frac{1}{(1+\varepsilon_n)^2}|\Theta'(t_n)|\leq |\Theta'(t)|\leq \frac{1}{(1-\varepsilon_n)^2}|\Theta'(t_n)|.
\end{equation}
It then follows from \eqref{eq2:Lem:technical-existence-set-Gn-real-spectrum} that 
$$
\frac{1}{1+\varepsilon_n}\leq \frac{\|k_t^\Theta\|_2}{\|k_{t_n}^\Theta\|_2}\leq \frac{1}{1-\varepsilon_n},
$$
and we get that $G_n$ satisfies condition $(ii)$. Condition $(i)$ is trivial and condition $(iii)$ follows along the same line as in \cite[Lemma 5.1]{baranov2005stability}. More precisely, using an increasing continuous branch of the argument of $\Theta$ on $G_n$ (note that $\sigma(\Theta)\cap G_n=\emptyset$), it is proved that for $t\in G_n$, we have
\begin{equation}\label{eq:estimate-fine-repro-baranov}
k_t^\Theta(t_n)\geq \frac{|\Theta'(t_n)|}{8\pi^2}.
\end{equation}
Now using the fact that 
$$
\sum_{n\geq 1}\frac{|k_t^\Theta(t_n)|^2}{|\Theta'(t_n)|}=\sum_{n\geq 1}|\langle k_t^\Theta,\kappa_{t_n}^\Theta\rangle|^2\leq C \|k_t^\Theta\|_2^2=C|\Theta'(t)|
$$
we see that the number of integers $n$ such that $t\in G_n$ is uniformly bounded. Hence $(iii)$ is satisfied. 

\end{proof}

\begin{remark}\rm{
It is natural to ask if Lemma~\ref{Lem:technical-existence-set-Gn-real-spectrum} is satisfied when we replace the inner function $\Theta$ by a general function $b$ in the unit ball of $H^\infty$. The difficulty is indeed to get the estimate \eqref{eq:estimate-fine-repro-baranov}.}

\end{remark}

\begin{theorem}\label{Thm:stability-real-frequencies}
Let $\Theta$ be an inner function, let $t_n\in\R$ such that $(\kappa_{t_n}^\Theta)_{n\geq 1}$ is a complete AOB in $K_\Theta$, and let $(s_n)_{n\geq 1}$ be a sequence of real numbers. Suppose there exists $N_0$ such that for all $n\geq N_0$, there is $\varepsilon_n>0$ verifying
\begin{equation}\label{Eq1-Thm:stability-real-frequencies}
\int_{[t_n,s_n]}\left(|\Theta'(t)|+|\Theta'(t)|^{-1}d_0^{-2}(t)\right)\,dt\leq \varepsilon_n
\end{equation}
or
\begin{equation}\label{Eq2-Thm:stability-real-frequencies}
|s_n-t_n|\leq \varepsilon_n |\Theta'(t_n)|\min(d_0^2(t_n),|\Theta'(t_n)|^{-2}),
\end{equation}
and $\lim_{n\to \infty}\varepsilon_n=0$. Then 
$(\kappa_{s_n}^\Theta)_{n\geq 1}$ is a complete AOB in $K_\Theta$. 
\end{theorem}

\begin{proof}
We can of course assume that $s_n\neq t_n$ and $\varepsilon_n<1/2$. Both 
\eqref{Eq1-Thm:stability-real-frequencies} and 
\eqref{Eq2-Thm:stability-real-frequencies} imply that there exists a point $u_n\in [s_n,t_n]$ such that 
$$
|s_n-t_n|\leq \varepsilon_n v_0(u_n).
$$
Then $v_0(u_n)\leq 4 v_0(t_n)$ and $|s_n-t_n|\leq 4 \varepsilon_n v_0(t_n)$. In particular, $s_n\in G_n$ where $G_n$ is defined in \eqref{eq:definition-Gn-real-spectrum} (replacing $\varepsilon_n$ by $4\varepsilon_n$). Moreover, 
using \eqref{eq:estimate1-kernel-Bernstein} and \eqref{eq:derivee-theta-perturbation}, we can write 
\begin{eqnarray*}
\frac{1}{\|k_{t_n}^\Theta\|_2^2}\int_{[t_n,s_n]}w_p^{-2}(z)\,|dz| \lesssim &  \int_{[t_n,s_n]}|\Theta'(t)|^{-1} \max(d_0^{-2}(t),|\Theta'(t)|^2)\,dt \\
\lesssim &\int_{[t_n,s_n]}(|\Theta'(t)|^{-1}d_0^{-2}(t)+|\Theta'(t)|)\,dt \lesssim \varepsilon_n.
\end{eqnarray*}
Then applying Lemma~\ref{Lem:technical-existence-set-Gn-real-spectrum} and  Theorem~\ref{prop1}, we get that $(\kappa_{s_n}^\Theta)_{n\geq 1}$ is an AOS. It remains to prove the completeness and the minimality. We argue as in the proof of Corollary~\ref{cor1-perturbation} replacing Lemma~\ref{Lem:add-minimality} by Lemma~\ref{Lem2:add-minimality}. More precisely, define 
$$
x_n=\begin{cases}
t_n & \hbox{if }n<N_0\\
s_n & \hbox{if }n\geq N_0
\end{cases},
$$
where $N_0$ will be choosen later. Then, we have
$$
\sup_{n\geq 1}\frac{1}{\|k_{t_n}^\Theta\|_2^2}\int_{[t_n,x_n]}w_p^{-2}(z)\,|dz| \lesssim \sup_{n\geq N_0}\varepsilon_n. 
$$
Thus we can find $N_0$ such that according to Theorem~\ref{prop1}, the sequence $(k_{x_n}^\Theta)_{n\geq 1}$ is a complete AOB in $K_\Theta$. Note that if $t_n\in\sigma(\Theta)$, then $v_0(t_n)=0$ and then $s_n=t_n$ and if $t_n\not\in\sigma(\Theta)$, then $G_n\subset \mathbb R\setminus\sigma(\Theta)$ and then $s_n\not\in\Theta$. Hence we can apply Lemma~\ref{Lem2:add-minimality} to get that $(\kappa_{s_n}^\Theta)_{n\geq 1}$ is minimal and complete in $K_\Theta$.  

\end{proof}

We also give an analogue of a result of Cohn \cite{cohn1986carleson} who studied small perturbations with respect to the change of the argument of the inner function $\Theta$. First, we need to introduce some more definition. An inner function $\Theta$ in $\mathbb{C}_+$ is said to be a \emph{meromorphic inner function} if it has a meromorphic extension to $\mathbb C$. In that case, we know that the argument of $\Theta$ is a real analytic increasing function on $\mathbb{R}$. Moreover, we say that an inner function $\Theta$ satisfies the \emph{connected level set condition} (abbreviated $\Theta\in (CLS)$) if there is $\delta\in (0,1)$ such that the set $\Omega(\Theta,\delta)=\{z\in\mathbb C_+:|\Theta(z)|<\delta\}$ is connected.  

\begin{corollary}\label{Cor:MIF-CLS-Stability-result}
Let $\Theta$ be a meromorphic inner function such that $\Theta\in (CLS)$, let $\varphi$ be its argument and let $t_n\in\R$ such that $(\kappa_{t_n}^\Theta)_{n\geq 1}$ is a complete AOB in $K_\Theta$. Let $(\varepsilon_n)_{n\geq 1}$ be a sequence of positive numbers tending to $0$. If
$$
|\varphi(s_n)-\varphi(t_n)|\leq \varepsilon_n,
$$
then $(\kappa_{s_n}^\Theta)_{n\geq 1}$ is a complete AOB in $K_\Theta$.
\end{corollary}

\begin{proof}
As noted in \cite[Remark 1, page 2419]{baranov2005stability}, since $\Theta$ is (CLS) and $(\kappa_{t_n}^\Theta)_n$ is an AOB, there exits a constant $C>0$ such that 
$$
|\Theta'(t)|^{-1}\leq C d_0(t),\qquad t\in G_n.
$$
Therefore 
\begin{eqnarray*}
\int_{[t_n,s_n]}(|\Theta'(t)|+|\Theta'(t)|^{-1}d_0^{-2}(t))\,dt \lesssim & \int_{[t_n,s_n]} |\Theta'(t)|\,dt\\
&=|\varphi(t_n)-\varphi(s_n)|\leq \varepsilon_n.
\end{eqnarray*}
Then apply Theorem~\ref{Thm:stability-real-frequencies}.
\end{proof}

\begin{example}\rm{
Let $\Theta_a(z)=e^{iaz}$, $a>0$, and fix $\alpha\in [0,2\pi)$. Then $\Theta_a^{-1}(\{e^{i\alpha}\})=\{t_n:n\in\mathbb{Z}\}$, with $t_n=(\alpha+2n\pi)/a$.  Then $(\kappa_{t_n}^{\Theta_a})_{n\in\Z}$ is an orthonormal basis of $K_{\Theta_a}$, the so--called Clark basis. Thus Corollary~\ref{Cor:MIF-CLS-Stability-result} says that if 
$$
\lim_{n\to \pm\infty}\left|s_n-\frac{\alpha+2n\pi}{a}\right|=0,
$$
then $(\kappa_{s_n}^{\Theta_a})_{n\in\Z}$ is a complete AOB for $K_{\Theta_a}$. }

\end{example}

\section{Example of exponential systems}
In the particular case where $\Theta_a(z)=e^{iaz}$, the Fourier transform $\mathcal F$ maps unitarily $K_{\Theta_a}$ onto $L^2(0,a)$ and $\mathcal F(\kappa_\lambda^{\Theta_a})=\chi_\lambda^a$, where 
$$
\chi^a_{\lambda}(t) = \left(\frac{2\Im \lambda}{1-e^{-2a\Im \lambda}}\right)^{1/2} e^{i\lambda t},\qquad \lambda\in\mathbb{C}_+.
$$
Thus, the geometric properties (completeness, minimality, Riesz basis, AOS, AOB,..) of system of normalized reproducing kernels $(\kappa_{\lambda_n}^{\Theta_a})_n$ in $K_{\Theta_a}$ and of normalized exponentials system $(\chi_{\lambda_n}^a)_n$ in $L^2(0,a)$ are the same. In \cite{chalendar2003functional}, AOS (or AOB) formed by reproducing kernels $k_{\lambda_n}^\Theta$ are studied under the additional condition that 
\begin{equation}\label{eq:Hypothese-hruvsvcev1981unconditional-sup}
\sup_{n\geq 1}|\Theta(\lambda_n)|<1.
\end{equation}
In the particular case when $\Theta=\Theta_a$, the condition 
\eqref{eq:Hypothese-hruvsvcev1981unconditional-sup} is equivalent to 
\begin{equation}\label{eq:condition-exponential-frequencies}
\inf_{n\geq 1}(\Im\lambda_n)>0.
\end{equation}
Under that assumption, it is proved in 
\cite[Proposition 7.2]{chalendar2003functional} that $(\chi_{\lambda_n}^a)_n$ is an AOB in $L^2(0,a)$ if and only if $(\lambda_n)_n$ is a thin sequence, which means that 
$$
\lim_{n\to \infty}\prod_{k\neq n}\left|\frac{\lambda_k-\lambda_n}{\lambda_k-\overline{\lambda_n}} \right|=1.
$$

Using Proposition~\ref{cns-general-AOB}, we construct a class of example of AOS where \eqref{eq:Hypothese-hruvsvcev1981unconditional-sup} (or equivalently \eqref{eq:condition-exponential-frequencies}) is not necessarily satisfied.

\begin{proposition}\label{prop:exponential-sup-egale1}
Let $(\lambda_n)_{n\geq 1} \subset \mathbb C$ be a sequence such that \\
(i) $\sup_n |\Im \lambda_n| < \infty$, \\
(ii) there exists a $q>1$ such that $\left|\frac{\lambda_{n+1}}{\lambda_n}\right|> q$ for all $n\geq 1$, \\
Then the sequence $(\chi^a_{\lambda_n})_{n\geq 1}$ is an AOS in $L^2(0,a)$ for all $a>0$.
\end{proposition}
\begin{proof}
We apply Proposition~\ref{cns-general-AOB}. Observe that
\begin{eqnarray*}
\Gamma_{n,m} = \langle \chi^a_{\lambda_n},\chi^a_{\lambda_m}\rangle = \left(\frac{4 \Im\lambda_n \Im\lambda_m}{(1- e^{-2a \Im\lambda_n})(1- e^{-2a \Im\lambda_m})}\right)^{1/2} \frac{e^{i(\lambda_n - \overline{\lambda_m})a}-1}{i(\lambda_n - \overline{\lambda_m})}
\end{eqnarray*}
and 
$$
\sup_{n,m\geq 1}\left|\frac{4 \Im\lambda_n \Im\lambda_m}{(1- e^{-2a \Im\lambda_n})(1- e^{-2a \Im\lambda_m})}\right| < \infty,
$$ 
provided $\sup_n\Im \lambda_n < \infty$. If $\Im\lambda_n=0$ (that is $\lambda_n\in\mathbb{R}$), the normalized factor $\Im\lambda_n/(1-e^{-2a\Im\lambda_n})$ should be understood as $a^{-1}$ and corresponds to $\|\chi_{\lambda_n}^a\|^2_{L^2(0,a)}=a$. It follows easily from $(ii)$ that for $m>n$, we have $|\lambda_m|>q^{m-n}|\lambda_n|$. Since $q>1$ that implies that $\lim_{n\to\infty}|\lambda_n|=\infty$. In particular, we can pick an integer $N$ such that for all $n\geq N$, we have $|\lambda_n|\geq 1$. Now, for $n\geq N$, write
\begin{eqnarray*}
\displaystyle\sum_{\substack{m\geq N\\m\neq n}}|\Gamma_{n,m}| \lesssim \displaystyle\sum_{\substack{m\geq N\\m\neq n}} \left|\frac{e^{i(\lambda_n - \overline{\lambda_m})a}-1}{i(\lambda_n - \overline{\lambda_m})}\right| &\lesssim& \displaystyle\sum_{N\leq m<n} \frac{1}{|\lambda_m| \left|\frac{\lambda_n}{\overline{\lambda_m}} -1 \right|} + \displaystyle\sum_{n<m} \frac{1}{|\lambda_n| \left|1- \frac{\overline{\lambda_m}}{\lambda_n} \right|}\\
&\leq& \displaystyle\sum_{N\leq m<n} \frac{1}{|\lambda_m| \left(\left|\frac{\lambda_n}{\lambda_m} \right| -1 \right)} + \displaystyle\sum_{n<m} \frac{1}{|\lambda_n| \left( \left|\frac{\lambda_m}{\lambda_n} \right|-1  \right)} \\
&\leq& \frac{1}{q-1} \displaystyle\sum_{N\leq m<n} \frac{1}{|\lambda_m|}+ \frac{1}{|\lambda_n|}\sum_{n<m}\frac{1}{q^{m-n}|\lambda_n|-1}\\
&\leq& \frac{1}{q-1}\frac{1}{|\lambda_N|}\displaystyle\sum_{N\leq m} \frac{1}{q^{m-N}}+\frac{1}{|\lambda_N|}\sum_{n<m}\frac{1}{q^{m-n}-1}.
\end{eqnarray*}
\noindent Thus, $\displaystyle\sup_{n \geq N} \sum_{\substack{m\geq N\\m\neq n}}|\Gamma_{n,m}| \lesssim \frac{1}{|\lambda_{N}|} \longrightarrow 0$, as $N \longrightarrow \infty$. Proposition \ref{cns-general-AOB} implies now that $(\chi^a_{\lambda_n})_{n\geq 1}$ is an AOS in $L^2(0,a)$.

\end{proof}

\begin{example}\rm{
The sequence $\lambda_n=r^n+i/n$ $,(r>1)$ satisfies the conditions of Proposition~\ref{prop:exponential-sup-egale1} and $\Im\lambda_n\to 0$ as $n$ goes to $\infty$. }
\end{example}

\section{Projecting onto a closed subspace}
Let $b_1,b_2\in H_1^\infty$ such that $b_2|b_1$, in the sense that $b_1=b_2b$ where $b\in H_1^\infty$. In this case, we know that $\HH(b_2)\subset\HH(b_1)$ and more precisely, we have 
\begin{equation}\label{eq:decomposition-non-directe}
\HH(b_1)=\HH(b_2)+b_2\HH(b).
\end{equation}
See \cite[I.10-I.11]{sarason1994sub} or \cite[Section 18.7]{fricain2014theory} for details on this decomposition. 

It should be noted that in general the above decomposition is not orthogonal. However for reproducing kernels, we do have such an orthogonal decomposition. 
\bl\label{Lem:decomposition-orthogonal-kernel}
Let $b_1=b_2b$ with $b_2,b\in H^\infty_1$. Let $\Lambda$ be a finite subset in $\mathbb C_+$. Then for any $a_\lambda\in\mathbb C$, $\lambda\in\Lambda$, we have 
\begin{equation}\label{eq:decomposition-orthogonal-kernel}
\left\|\sum_{\lambda\in\Lambda}a_\lambda k_{\lambda}^{b_1}\right\|_{b_1}^2=\left\|\sum_{\lambda\in\Lambda}a_\lambda k_{\lambda}^{b_2}\right\|_{b_2}^2+\left\|\sum_{\lambda\in\Lambda}a_\lambda \overline{b_2(\lambda)}k_{\lambda}^{b}\right\|_{b}^2
\end{equation}
\el
\begin{proof}
First note that
\begin{equation}\label{eq:decomposition-kernel}
k_{\lambda}^{b_1}=k_{\lambda}^{b_2}+b_2\overline{b_2(\lambda)}k_{\lambda}^b.
\end{equation}
Now if $LH$ and $RH$ denotes the left hand-side and right hand-side of \eqref{eq:decomposition-orthogonal-kernel}, we have
$$
LH=\sum_{\lambda,\mu\in\Lambda}a_\lambda \overline{a_\mu}k_\lambda^{b_1}(\mu),
$$
and 
$$
RH=\sum_{\lambda,\mu\in\Lambda}a_\lambda \overline{a_\mu}k_\lambda^{b_2}(\mu)+\sum_{\lambda,\mu\in\Lambda}a_\lambda \overline{a_\mu}\overline{b_2(\lambda)}b_2(\mu) k_\lambda^{b}(\mu).
$$
It remains to use \eqref{eq:decomposition-kernel} to get \eqref{eq:decomposition-orthogonal-kernel}.

\end{proof}

Let $(\lambda_n)_{n\geq 1}\subset\mathbb C_+$ and assume that $(\kappa_{\lambda_n}^{b_1})_{n\geq 1}$ is a complete AOB in $\HH(b_1)$. It is very natural to ask if the sequence $(\kappa_{\lambda_n}^{b_2})_{n\geq 1}$ remains an AOB in $\HH(b_2)$. The anwer depends on the following ratio:
$$\RR_{b_1,b_2}(n) := \frac{\|k^{b_1}_{\lambda_n}\|_{b_1}^2}{\|k^{b_2}_{\lambda_n}\|_{b_2}^2} =\frac{1-|b_1(\lambda_n)|^2}{1- |b_2(\lambda_n)|^2}.
$$
The following result says that if the behavior of 
$b_1(\lambda_n)$ and $b_2(\lambda_n)$ are comparable as $n \rightarrow \infty$, then we can transfer AOBs between the respective de Branges--Rovnyak spaces. 
\begin{theorem}\label{Thm:cns-division}
Let $b_1=b_2b$, $b_2,b\in H^\infty_1$, let $(\lambda_n)_{n\geq 1}\subset\mathbb C_+$ satisfying 
$$
\sum_n\left|\RR_{b_1,b_2}(n)-1\right|<\infty.
$$
If the sequence $(\kappa_{\lambda_n}^{b_1})_{n\geq 1}$ is a complete AOB in $\HH(b_1)$, then there is an integer $p\geq 1$ such that $(\kappa_{\lambda_n}^{b_2})_{n\geq p}$ is a complete AOB in $\HH(b_2)$. Conversely, if $(\kappa_{\lambda_n}^{b_2})_{n\geq 1}$ is an AOB in $\HH(b_2)$, then $(\kappa_{\lambda_n}^{b_1})_{n\geq 1}$ is an AOB in $\HH(b_1)$. 
\end{theorem}
\begin{proof}
First note that $(k_{\lambda_n}^{b_2})_{n\geq 1}$ is complete in $\HH(b_2)$. Indeed, let $f\in\HH(b_2)$, $f\perp k_{\lambda_n}^{b_2}$, $n\geq 1$. Since $\HH(b_2)\subset\HH(b_1)$, we can write 
$$
0=\langle f,k_{\lambda_n}^{b_2} \rangle_{b_2}=f(\lambda_n)=\langle f,k_{\lambda_n}^{b_1} \rangle_{b_1}.
$$
Thus $f$ is orthogonal to $k_{\lambda_n}^{b_1}$, $n\geq 1$ and the completeness of $(k_{\lambda_n}^{b_1})_{n\geq 1}$ in $\HH(b_1)$ implies that $f\equiv 0$. 

Since $(\kappa_{\lambda_n}^{b_1})_{n\geq 1}$ is an AOB in $\HH(b_1)$, 
given any $\epsilon > 0$, there exists $N \in \mathbb N$ such that 
\begin{equation}
\label{proj1} (1- \epsilon) \sum_{n \geq N} |a_n|^2 \leq  \left\|\displaystyle\sum_{n \geq N} a_n \kappa^{b_1}_{\lambda_n}\right\|_{b_1} \leq (1 + \epsilon)\sum_{n \geq N} |a_n|^2.
\end{equation}
Moreover, since $\RR_{b_1,b_2}(n) -1 \in \ell ^1$, we can also assume that $N$ satisfies 
\begin{equation} \label{proj3}
\displaystyle\sum_{n \geq N} \left(\frac{\|k^{b_1}_{\lambda_n}\|_{b_1}^2}{\|k^{b_2}_{\lambda_n}\|_{b_2}^2} - 1 \right) < \epsilon.
\end{equation}
In particular, this guarantees that
\begin{equation}
\label{proj2} 1- \epsilon < \frac{\|k^{b_1}_{\lambda_n}\|_{b_1}^2}{\|k^{b_2}_{\lambda_n}\|_{b_2}^2} < 1+\epsilon.
\end{equation}
We now prove that $\{\kappa^{b_2}_{\lambda_n}\}$ is an AOS in $\HH(b_2)$. Using Lemma~\ref{Lem:decomposition-orthogonal-kernel}, we have
$$
\left\|\displaystyle\sum_{n \geq N} a_n \frac{k^{b_1}_{\lambda_n}}{\|k^{b_2}_{\lambda_n}\|_{b_2}}\right\|^2_{b_1} = \left\|\displaystyle\sum_{n \geq N}a_n \frac{ k^{b_2}_{\lambda_n}}{\|k^{b_2}_{\lambda_n}\|_{b_2}} \right\|_{b_2}^2 + \left\|\displaystyle\sum_{n \geq N}a_n \frac{ \overline{b_2(\lambda_n)}k^{b}_{\lambda_n}}{\|k^{b_2}_{\lambda_n}\|_{b_2}}\right\|_b^2.
$$
Thus
\begin{eqnarray*}
\left\|\displaystyle\sum_{n \geq N}a_n \kappa^{b_2}_{\lambda_n} \right\|_{b_2}^2 &=& \left\|\displaystyle\sum_{n \geq N} a_n \frac{\|k^{b_1}_{\lambda_n}\|_{b_1}}{\|k^{b_2}_{\lambda_n}\|_{b_2}} \kappa^{b_1}_{\lambda_n}\right\|_{b_1}^2 - \left\|\displaystyle\sum_{n \geq N} a_n \frac{ \overline{b_2(\lambda_n)} k^{b}_{\lambda_n}}{\|k^{b_2}_{\lambda_n}\|_{b_2}}\right\|_b^2.\\
&=& I_1 - I_2.
\end{eqnarray*}
For $I_1$, we use relationships \eqref{proj1} and \eqref{proj2} to get
\begin{eqnarray*}
(1-\epsilon) ^2 \displaystyle\sum_{n \geq N} |a_n|^2 \leq (1-\epsilon)\displaystyle\sum_{n \geq N} |a_n|^2  \frac{\|k^{b_1}_{\lambda_n}\|_{b_1}^2}{\|k^{b_2}_{\lambda_n}\|_{b_2}^2}  &\leq& \left\|\displaystyle\sum_{n \geq N} a_n \frac{\|k^{b_1}_{\lambda_n}\|_{b_1}}{\|k^{b_2}_{\lambda_n}\|_{b_2}}\kappa^{b_1}_{\lambda_n}\right\|_{b_1}^2 \\  &\leq &(1+ \epsilon)\displaystyle\sum_{n \geq N} |a_n|^2 \frac{\|k^{b_1}_{\lambda_n}\|_{b_1}^2}{\|k^{b_2}_{\lambda_n}\|_{b_2}^2} \leq (1+\epsilon)^2 \displaystyle\sum_{n \geq N} |a_n|^2.
\end{eqnarray*}
For $I_2$, we use \eqref{eq:decomposition-orthogonal-kernel}, \eqref{proj3} and Cauchy--Schwarz inequality to obtain
\begin{eqnarray*}
\left\|\displaystyle\sum_{n \geq N} a_n \frac{\overline{b_2(\lambda_n)}k^{b}_{\lambda_n}}{\|k^{b_2}_{\lambda_n}\|_{b_2}}\right\|_b^2 &\leq& \left(\displaystyle\sum_{n\geq N} |a_n|^2\right) \left(\displaystyle\sum_{n \geq N} \frac{\| \overline{b_2(\lambda_n)}  k^{b}_{\lambda_n}\|_b^2}{\|k^{b_2}_{\lambda_n}\|_{b_2}^2}\right)\\ 
&=& \left(\displaystyle\sum_{n\geq N} |a_n|^2\right)
\displaystyle\sum_{n \geq N} \left( \frac{\|k^{b_1}_{\lambda_n} \|_{b_1}^2}{\|k^{b_2}_{\lambda_n}\|_{b_2}^2} -1\right) \\
&\leq& \epsilon \displaystyle\sum_{n\geq N} |a_n|^2
\end{eqnarray*}
It follows that $(\kappa^{b_2}_{\lambda_n})_n$ is an AOS. Now let $p$ be the smallest integer such that $(\kappa_{\lambda_n}^{b_2})_{n\geq p}$ is an AOB in $\HH(b_2)$. If $p=1$, then since $(\kappa_{\lambda_n}^{b_2})_{n\geq 1}$ is complete in $\HH(b_2)$, we have the result. Otherwise combining Lemma~\ref{Lem:add-minimality} (b) and the fact that a sequence is an AOB if and only if it is a minimal AOS, we conclude that$(\kappa_{\lambda_n}^{b_2})_{n\geq p}$ is a complete AOB in $\HH(b_2)$.\\


Conversely, suppose $(\kappa^{b_2}_{\lambda_n})_{n\geq 1}$ is an AOB in $\HH(b_2)$. We note that $(\RR_{b_2,b_1}(n)-1)_n=(1/\RR_{b_1,b_2}(n) -1)_n \in \ell^1$. Using similar computations as before, we easily see that $(\kappa^{b_1}_{\lambda_n})_n$ is an AOS in $\HH(b_1)$. It remains to check the minimality of $(k^{b_1}_{\lambda_n})_{n\geq 1}$. But, since $(k^{b_2}_{\lambda_n})_{n\geq 1}$ is minimal in $\HH(b_2)$, there exists a sequence of functions $\psi_n\in \HH(b_2)$ such that 
$$
\langle \psi_n,k_{\lambda_\ell}^{b_2}\rangle_{b_2}=\delta_{n,\ell}.
$$
Remind now that $\HH(b_2)\subset\HH(b_1)$, whence
$$
\langle \psi_n,k_{\lambda_\ell}^{b_1} \rangle_{b_1}=\psi_n(\lambda_\ell)=\langle \psi_n,k_{\lambda_\ell}^{b_2}\rangle_{b_2}=\delta_{n,\ell},
$$
which proves that $(k^{b_1}_{\lambda_n})_{n\geq 1}$ is a minimal sequence in $\HH(b_1)$. 
\end{proof}

\begin{corollary}
Let $b_1$ and $b_2$ be two functions in $H^\infty_1$ such that they have a common factor $b$, i.e. both $b_1/b$ and $b_2/b$ are in $H^\infty_1$. Moreover, suppose that $(\RR_{b_1,b}(n)-1)_n \in \ell^1$ and $(\RR_{b_2,b}(n) - 1)_n \in \ell^1$. If $(\kappa^{b_1}_{\lambda_n})_{n\geq 1}$ is an AOB in $\HH(b_1)$, then there is an integer $p\geq 1$ such that $(\kappa^{b_2}_{\lambda_n})_{n\geq p}$ is an AOB in $\HH(b_2)$.
\end{corollary}

The assumption that $(\RR_{b_1,b_2}(n)-1)_n\in \ell^1$ may appear very restrictive. However, as the following result shows, in some particular case, it is indeed also necessary. 

\begin{corollary}
Let $b_1=\Theta_2 b$ where $b\in H_1^\infty$ and $\Theta_2$ is an inner function such that $\infty\notin\sigma(\Theta_2)$. Let $(\lambda_n)_{n\geq 1}$ be a sequence of points in $\mathbb C_+ $ such that $(\kappa^{b_1}_{\lambda_n})_{n\geq 1}$ is a complete AOB in $\HH(b_1)$ and
\begin{equation}\label{eq:norme-bounded}
\sup_{n\geq 1}\|k_{\lambda_n}^b\|_b<\infty.
\end{equation} 
Then the following are equivalent:
\begin{enumerate}
\item[$(1)$] There is an integer $p\geq 1$ such that $(\kappa^{\Theta_2}_{\lambda_n})_{n\geq p}$ is a complete AOB in $K_{\Theta_2}$. 
\item[$(2)$] $(\RR_{b_1,\Theta_2}(n) - 1)_n \in \ell^1$.
\end{enumerate}
\end{corollary}


\begin{proof}
$(2) \implies (1)$  follows from Theorem~\ref{Thm:cns-division}.

\noindent  $(1) \implies (2)$ We recall a well known fact (see \cite[Lemma 4.4]{baranov2005stability}) that $\sup_{n} |\lambda_n| < \infty$, provided $\infty \not\in \sigma(\Theta_2)$ and $(\kappa^{\Theta_2}_{\lambda_n})_{n\geq p}$ is an AOB in $K_{\Theta_2}$ (in fact, it is sufficient that $(\kappa^{\Theta_2}_{\lambda_n})_n$ to be a frame). 

Let $\gamma \in \mathbb C_+$. Then, the function 
$$ f(z):= \Theta_2(z) \frac{1-\overline{b(\gamma)}b(z)}{z-\overline \gamma} \in \Theta_2 \HH(b) \subset K_{\Theta_2} + \Theta_2 \HH(b) =\HH(b_1).$$
Since $(\kappa^{b_1}_{\lambda_n})_n$ is an AOB in $\HH(b_1)$, it must be the case that 
\begin{eqnarray*}
&\displaystyle\sum_{n \geq 1}& |\langle f, \kappa^{b_1}_{\lambda_n}\rangle|^2 < \infty \\
\mbox{i.e. } \hspace{20 pt}&\displaystyle\sum_{n \geq 1}& |\Theta_2(\lambda_n)|^2 \left|\frac{1-\overline{b(\gamma)}b(\lambda_n)}{\lambda_n-\overline \gamma}\right|^2 \frac{2\Im \lambda_n}{1-|b_1(\lambda_n)|^2} < \infty.
\end{eqnarray*}
We observe that, since $\sup_n |\lambda_n|<\infty$, when $|\gamma|$ is large enough, we have
\begin{equation*}
\left|\frac{1-\overline{b(\gamma)}b(\lambda_n)}{\lambda_n-\overline \gamma}\right| \gtrsim \frac{1-|b(\gamma)|}{|\gamma|}.
\end{equation*}
Thus, 
$$ \displaystyle\sum_{n \geq 1} |\Theta_2(\lambda_n)|^2  \frac{2\Im \lambda_n}{1-|b_1(\lambda_n)|^2} <\infty.$$
Since $1-|b(\lambda_n)|^2\lesssim \Im\lambda_n$, we have
$$ \displaystyle\sum_{n \geq 1} |\Theta_2(\lambda_n)|^2  \frac{1-|b(\lambda_n)|^2}{1-|b_1(\lambda_n)|^2} < \infty.$$
i.e
$$\displaystyle\sum_{n \geq 1} |\Theta_2(\lambda_n)|^2  \frac{\|k^b_{\lambda_n}\|_b^2}{\|k^{b_1}_{\lambda_n}\|_{b_1}^2} < \infty. $$
Thus, we have finally, $$\displaystyle\sum_{n \geq 1}(1 - \RR_{\Theta_2,b_1}(n))= \displaystyle\sum_{n \geq 1} \frac{\|k^{b_1}_{\lambda_n}\|_{b_1}^2-\|k^{\Theta_2}_{\lambda_n}\|_{2}^2}{\|k^{b_1}_{\lambda_n}\|_{b_1}^2} < \infty. $$
In other words, $(\RR_{\Theta_2,b_1}(n) - 1)_n \in \ell^1.$ Since $\RR_{b_1,\Theta_2}(n) = 1/ \RR_{\Theta_2,b_1}(n)$, it follows that $(\RR_{b_1,\Theta_2}(n) -1)_n \in \ell^1$.

\end{proof}

\begin{example}\rm{
Note that \eqref{eq:norme-bounded} is, in particular, satisfied in the case when $b=\Theta$ is an inner function satisfying $\Theta'\in L^\infty(\mathbb R)$. Indeed, as was shown in \cite[Corollary 4.7]{baranov2005bernstein}, we have
$$
\|k_{\lambda_n}^\Theta\|_2\leq \|k_{x_n}^\Theta\|_2=|\Theta'(x_n)|^{1/2},
$$
where $x_n=\Re\lambda_n$. }
\end{example}

\begin{remark}\rm{
The results given in that section can also be proved when $b_1=\Theta_1$ is an inner function and the frequencies $(\lambda_n)_{n\geq 1}$ belong to $\mathbb{C}_+\cup\mathbb R\setminus\sigma(\Theta)$. }
\end{remark}
\nopagebreak 

\end{document}